\documentclass{amsart}

\usepackage[T1]{fontenc}
\usepackage[utf8]{inputenc}
\usepackage{lmodern}
\usepackage{amssymb,amsmath,mathtools,mathrsfs}
\usepackage{microtype}
\usepackage{booktabs}
\usepackage{needspace}
\usepackage{xcolor}
\usepackage{graphicx}
\usepackage{tikz,tikz-3dplot}
\usetikzlibrary{calc}
\usepackage[hidelinks]{hyperref}
\hypersetup{pdftitle={Examples of absolutely minimal polyhedral chains in maximum norm spaces},
	pdfauthor={Roger Z\"ust}}

\numberwithin{equation}{section}
\newtheorem{Thm}{Theorem}[section]
\newtheorem{Prop}[Thm]{Proposition}
\newtheorem{Lem}[Thm]{Lemma}
\newtheorem{Cor}[Thm]{Corollary}
\newtheorem{Que}[Thm]{Question}
\theoremstyle{definition}
\newtheorem{Expl}[Thm]{Example}
\newtheorem{Rem}[Thm]{Remark}

\newcommand{\dist}{\operatorname{dist}}

\newcommand{\conv}{\operatorname{conv}}
\newcommand{\spt}{\operatorname{spt}}
\newcommand{\Fill}{\operatorname{Fill}}
\newcommand{\Lip}{\operatorname{Lip}}
\newcommand{\Area}{\operatorname{Area}}
\newcommand{\id}{\mathrm{id}}
\newcommand{\R}{\mathbf R}

\newcommand{\Z}{\mathbf Z}
\newcommand{\N}{\mathbf N}
\newcommand{\cL}{\mathcal L}

\newcommand{\bM}{\mathbf M}
\newcommand{\bI}{\mathbf I}
\newcommand{\bN}{\mathbf N}
\newcommand{\bone}{\mathbf 1}
\newcommand{\curr}[1]{[\![#1]\!]}
\newcommand{\norm}[1]{\lVert#1\rVert}

\definecolor{diskblue}{RGB}{49,113,145}
\definecolor{conegold}{RGB}{211,153,66}
\definecolor{tipcream}{RGB}{215,185,151}
\definecolor{rectcoral}{RGB}{211,113,99}
\definecolor{collarteal}{RGB}{71,145,132}
\definecolor{heatnone}{RGB}{238,241,243}
\definecolor{heathalf}{RGB}{105,179,190}
\definecolor{heatfull}{RGB}{23,79,108}
\tikzset{surfaceedge/.style={draw=black!55,line width=.25pt},
	boundaryedge/.style={draw=black!85,line width=.7pt}}

\newcommand{\tickminus}{\mathord{\scalebox{.65}[1]{$-$}}}
\newcommand{\projectiongrid}{%
	\foreach \g in {-1.2,-.8,.8,1.2}{%
		\draw[black!35,line width=.3pt,dash pattern=on 1.5pt off 1.2pt]
		(\g,-2)--(\g,2);
		\draw[black!35,line width=.3pt,dash pattern=on 1.5pt off 1.2pt]
		(-2,\g)--(2,\g);
	}%
}
\newcommand{\projectionframe}[2]{%
	\draw[black!55,line width=.35pt] (-2,-2) rectangle (2,2);
	\draw[->,black!60,line width=.35pt] (-2.06,-2)--(2.22,-2)
	node[right,inner sep=2pt] {$#1$};
	\draw[->,black!60,line width=.35pt] (-2,-2.06)--(-2,2.22)
	node[above,inner sep=2pt] {$#2$};
	\foreach \v/\lab in {-2/2,-1.2/{\frac65},-.8/{\frac45},0/0,.8/{\frac45},1.2/{\frac65},2/2}{%
		\draw[black!60,line width=.35pt] (\v,-2)--(\v,-2.06);
		\draw[black!60,line width=.35pt] (-2,\v)--(-2.06,\v);
		\ifdim\v pt<0pt
		\node[anchor=north,inner sep=2pt,font=\scriptsize] at (\v,-2.06)
		{$\vphantom{\frac65}\mathllap{\tickminus\mkern-1mu}\lab$};
		\node[anchor=center,inner sep=0pt,xshift=-5pt,font=\scriptsize] at (-2.06,\v)
		{$\mathllap{\tickminus\mkern-1mu}\lab$};
		\else
		\node[anchor=north,inner sep=2pt,font=\scriptsize] at (\v,-2.06)
		{$\vphantom{\frac65}\lab$};
		\node[anchor=center,inner sep=0pt,xshift=-5pt,font=\scriptsize] at (-2.06,\v)
		{$\lab$};
		\fi
	}%
}

\title[Minimal polyhedral chains in maximum norm spaces]{Examples of absolutely minimal polyhedral chains in maximum norm spaces}
\author{Roger Z\"ust}
\email{roger\_zuest@hotmail.com}
\date{\today}

\begin{document}
	\begin{abstract}
		We construct explicit area-minimizing polyhedral disks in $\R^n$ with the maximum norm for $n=2,3,4$. Their boundaries are polygonal images of the geodesic circle of length $2\pi$, obtained from $n$ distance functions. Piecewise constant calibrations show that these disks minimize mass among Ambrosio--Kirchheim metric currents with the same boundary. The example for $n=4$ gives the lower bound $11\pi^2/20\approx5.42828$ for Gromov's filling area conjecture. For $n>4$, finite linear programs provide further lower bounds.
	\end{abstract}
	\maketitle
	
	\section{Introduction}
	
	Gromov's filling area conjecture~\cite{G} states that a compact orientable Riemannian surface has area at least $2\pi$ if its boundary is isometric to the geodesic circle of length $2\pi$, with distances measured in the surface. Equality holds for the unit hemisphere.
	
	Using distance functions to $n$ points spaced $\pi/n$ apart on a semicircle, we obtain a $1$-Lipschitz map of the circle into $\ell^n_\infty=(\R^n,\norm{\cdot}_\infty)$. Its image is a polygon $S_n$ with $2n$ vertices. For $n=2,3,4$, we construct embedded polyhedral disks $D_n$
	spanning $S_n$ and prove
	\begin{equation}\label{eq:main-values}
		\bM(D_n)=\Fill(S_n)=
		\begin{cases}
			\frac12\pi^2, & n=2,3,\\
			\frac{11}{20}\pi^2, & n=4.
		\end{cases}
	\end{equation}
	Here $\Fill(S_n)$ denotes the least mass among metric
	currents with boundary $S_n$, in the sense of Ambrosio and
	Kirchheim~\cite{AK}.
	
	The maximum norm places these examples in a general setting for metric filling problems. Fr\'echet~\cite[pp.~161--162]{Fr} showed that every separable metric space embeds isometrically into $\ell_\infty=\ell_\infty(\N)$. This space is injective: every $1$-Lipschitz map from a subset of a metric space into $\ell_\infty$ extends to the whole space with the same Lipschitz constant. Together, these properties make $\ell_\infty$ a universal ambient space for metric filling problems in the sense of Gromov~\cite{G}, for any Finsler volume. Polyhedral approximation reduces integral filling problems with
	compactly supported boundary to problems with polyhedral boundary in $\ell^n_\infty$, up to arbitrarily small errors in the least
	filling mass; see Section~\ref{sec:currents}.
	
	The finite-dimensional spaces $\ell^n_\infty$ are also injective. Thus the $1$-Lipschitz map from the circle onto $S_n$ extends to every filling of the circle without increasing mass. In particular, the bound $11\pi^2/20\approx5.42828$ holds for metric fillings of the geodesic circle of length $2\pi$; see Corollary~\ref{cor:circle-currents}. It also bounds the area of every Riemannian filling, including nonorientable ones; see Corollary~\ref{cor:lower_bound}. For comparison, Briggs and Wells~\cite[Theorem~1.2]{BW} obtain $\sqrt3\,\pi^2/4$ by a discrete method. Chen, Li and Zhong~\cite[Theorems~1.1--1.2]{CLZ} obtain $14\zeta(3)/\pi\approx5.35677$ for arbitrary Riemannian fillings and $5.40154$ for orientable fillings.
	
	On rectifiable surfaces, Ambrosio--Kirchheim mass is induced by mass$^\ast$ volume. In $\ell^n_\infty$, the mass$^\ast$ area of a flat piece is the largest Lebesgue area of its coordinate projections. This formula makes the masses of our disks explicit and guides the construction of the piecewise constant calibrations proving the lower bounds. It suffices to verify the calibration inequalities on the injective hull of $S_n$, since fillings can be retracted onto it without increasing mass; see Proposition~\ref{prop:hull}. Each is a sum of coordinate forms $dx_i\wedge dx_j$ with coefficients depending only on $x_i$ and $x_j$; see Proposition~\ref{prop:calibration} and the calibration for $n=4$ in \eqref{eq:n4-form}. The same argument applies to rectifiable chains with coefficients in any complete normed Abelian group; see Remark~\ref{rem:coefficients}. For $n>4$, the search for these calibrations leads to the linear
	programs in Section~\ref{sec:lp}. These yield improved lower bounds
	for the filling area of the circle, including a bound exceeding
	$5.5$ for $n=6$. We ask whether these bounds can approach the
	conjectured value $2\pi$ arbitrarily closely.
	
	\begin{Que}\label{que:sharp-bound}
		For every $\varepsilon>0$, do the calibrations of Proposition~\ref{prop:calibration} give a lower bound $2\pi-\varepsilon$ for $\Fill(S_n)$ for some $n$?
	\end{Que}
	
	The polyhedral form of our minimizers also suggests a regularity question. In Euclidean space, the regularity of area-minimizing integral currents has been studied since the work of Federer and Fleming~\cite{FF} and remains an active subject; see the survey~\cite{DL}. By Almgren's regularity theorem, an $m$-dimensional area-minimizing integral current is smooth in the interior away from a singular set of Hausdorff dimension at most $m-2$; see~\cite{Alm} and \cite[Theorem~0.3]{DLS}. Our examples raise the question of whether polyhedral minimizers
	for mass$^\ast$ in $\ell^n_\infty$ exist more generally.
	
	Write $\mathscr P_m$ for polyhedral $m$-chains with integer coefficients and $\bI_m$ for integral currents. We ask the following.
	
	\begin{Que}\label{que:polyhedral}
		Let $1\le m\le n$ and let $S\in\mathscr P_{m-1}(\ell^n_\infty)$ be a cycle, with total coefficient zero when $m=1$.
		Does there exist $P\in\mathscr P_m(\ell^n_\infty)$ satisfying $\partial P=S$ and
		\[
		\bM(P)=\min\{\bM(T):T\in\bI_m(\ell^n_\infty),\ \partial T=S\}?
		\]
	\end{Que}
	
	Section~\ref{sec:currents} introduces currents and mass, explains the properties of mass$^\ast$ used below, and justifies the finite-dimensional approximation. Section~\ref{sec:nonunique} illustrates nonuniqueness. Sections~\ref{sec:discrete} and~\ref{sec:lp} give the explicit constructions and the linear programs. Finally, Section~\ref{sec:polyhedral} discusses Question~\ref{que:polyhedral}, including the positive answers for $m=1$ and $m=n$.
	
	\section{Currents and mass}\label{sec:currents}
	
	\subsection{Currents}\label{subsec:currents}
	
	We write $\bN_m(X)$ and $\bI_m(X)$ for normal and integral currents of dimension $m\ge0$ in a complete metric space $X$, in the sense of Ambrosio and Kirchheim~\cite{AK}. Their mass and mass measure are denoted by $\bM(T)$ and $\|T\|$. We write $\curr P$ for the current of integration over an oriented polyhedron $P$. For a cycle $S\in\bN_m(X)$, define
	\[
	\Fill_X(S)\coloneqq\inf\{\bM(T):T\in\bN_{m+1}(X),\ \partial T=S\}.
	\]
	The ambient space is omitted when clear.
	
	If $f:X\to Y$ is Lipschitz and $T\in\bN_m(X)$, then
	$\partial(f_\#T)=f_\#(\partial T)$ and
	\begin{equation}\label{eq:local-pushforward}
		\|f_\#T\|\le\Lip(f)^m f_\#\|T\|.
	\end{equation}
	In particular, $1$-Lipschitz maps do not increase mass.
	In $\R^n$, compactly supported metric normal currents identify canonically with compactly supported classical normal currents in the sense of Federer--Fleming~\cite{FF}; see~\cite[Theorem~11.1]{AK} and \cite[Theorem~5.5]{L1}. This identification commutes with the boundary and is unchanged if $\R^n$ is equipped with another norm.
	
	For chains with coefficients in a complete normed Abelian group $G$, we use the rectifiable chains in the theory of De Pauw--Hardt~\cite{DPH}; see also White~\cite{W1,W2} in Euclidean spaces and Adams~\cite{Ad} in Banach spaces. Rectifiable currents with real or integer multiplicities identify with the corresponding rectifiable chains of De Pauw--Hardt~\cite[\S3.6(2)]{DPH}. We equip these chains with the mass induced by mass$^\ast$ volume, as described below; under this identification, their mass agrees with the Ambrosio--Kirchheim mass.
	
	\subsection{Mass}\label{subsec:mass}
	
	A \emph{Finsler $m$-volume} assigns to each $m$-dimensional normed
	space $V$ a Haar measure $\nu_V$, and agrees with ordinary volume
	when the norm is Euclidean. It is monotone under linear contractions:
	if $L:V\to W$ is linear and $1$-Lipschitz between $m$-dimensional
	normed spaces, then
	\[
	\nu_W(L(A))\le\nu_V(A)
	\qquad\text{for every compact }A\subset V.
	\]
	See~\cite[Definition~3.2]{APT}. The Busemann definition~\cite{Bus}
	induces Hausdorff measure with its standard normalization;
	see~\cite[\S3.1 and Theorem~3.23]{APT}.
	
	A Finsler volume induces a mass on rectifiable currents by
	integrating the corresponding volume density weighted by the
	absolute multiplicity. The same construction applies to rectifiable
	$G$-chains, using the norm of the coefficient;
	see~\cite[\S2.1]{Z2}.
	We use mass$^\ast$, introduced by Benson~\cite{Ben} and used by
	Gromov~\cite{G}, and denote it by $\mu_V^{\mathrm{m}\ast}$ or simply
	$\mu$. Its induced mass on rectifiable currents is exactly the
	Ambrosio--Kirchheim mass $\bM$;
	see~\cite[\S9]{AK} and~\cite[Lemma~2.5(2)]{Z2}.
	
	When $\dim V=m$, this volume is equivalently a norm on the one-dimensional space $\Lambda_m V$. The $m$-vector $\xi\coloneqq v_1\wedge\cdots\wedge v_m$ represents the parallelepiped
	$P\coloneqq\{\sum_{i=1}^m t_i v_i:0\le t_i\le1\}\subset V$ and
	\[
	\mu(P)=\mu(\xi)
	=\sup_{\norm{\alpha_1}_*,\ldots,\norm{\alpha_m}_*\le1}
	|(\alpha_1\wedge\cdots\wedge\alpha_m)(\xi)|,
	\]
	where $\norm{\cdot}_*$ is the dual norm on $V^*$.
	For a simple $m$-vector in a larger ambient space, we use the same notation for the volume in its span.
	
	Let $\Lambda(n,m)$ be the increasing $m$-tuples in $\{1,\ldots,n\}$, write
	\[
	dx_\lambda\coloneqq dx_{\lambda_1}\wedge\cdots\wedge dx_{\lambda_m},
	\]
	and let $\pi_\lambda$ be the corresponding coordinate projection. The dual unit ball of $\ell^n_\infty$ is the convex hull of $\pm dx_j$. Multilinearity and Hahn--Banach extension of covectors therefore give
	\begin{equation}\label{eq:mass-density}
		\mu(\xi)=\max_{\lambda\in\Lambda(n,m)}|dx_\lambda(\xi)|.
	\end{equation}
	See also~\cite[Lemma~5.1]{Zfrac} for the corresponding coordinate formula
	for the mass of currents in $\ell_\infty$.
	
	For any Borel set $A$ contained in an $m$-dimensional affine plane, this gives
	\[
	\mu(A)=\max_{\lambda\in\Lambda(n,m)}\cL^m(\pi_\lambda(A)).
	\]
	Indeed, on a fixed plane all coordinate projections have constant
	Jacobians, and one of them realizes the mass$^\ast$ density.
	In particular, the mass of an oriented simplex in $\ell^n_\infty$ is
	\begin{equation}\label{eq:simplex-mass}
		\bM\bigl(\curr{p_0,\ldots,p_m}\bigr)
		=\frac1{m!}\mu\bigl((p_1-p_0)\wedge\cdots\wedge(p_m-p_0)\bigr).
	\end{equation}
	The volume of an arbitrary polyhedron is therefore computed by a sum of maximal absolute minors after triangulation.
	
	For comparison, the Cauchy--Binet identity in Euclidean $\R^n$ gives
	\[
	|\xi|_2^2
	=\det(\langle v_i,v_j\rangle)_{i,j=1}^m
	=\sum_{\lambda\in\Lambda(n,m)}|dx_\lambda(\xi)|^2.
	\]
	Thus the Euclidean density takes the $2$-norm of the minors, whereas the mass$^\ast$ density in~\eqref{eq:mass-density} takes their maximum absolute value.
	
	Mass$^\ast$ is \emph{totally convex}: for each nonzero simple $m$-vector $\xi$ there is a decomposable $m$-covector $\eta$
	such that
	\begin{equation}\label{eq:total-convexity}
		\eta(\xi)=\mu(\xi),\qquad
		|\eta(\zeta)|\le\mu(\zeta)
		\quad\text{for every simple $m$-vector $\zeta$.}
	\end{equation}
	Equivalently, every $m$-plane $V$ admits a linear projection $p_V$ onto $V$
	with $p_V|_V=\id_V$ and
	$\mu(\bigwedge^m p_V(\zeta))\le\mu(\zeta)$.
	This is contraction of $m$-volume; it does not require $\Lip(p_V)\le1$. Mass$^\ast$ has this property in every finite-dimensional normed space
	\cite[Definition~4.26 and Theorems~4.27--4.28]{APT}. In $\ell^n_\infty$, choose a signed coordinate form attaining the maximum in \eqref{eq:mass-density}; then
	$p_V\coloneqq(\pi_\lambda|_V)^{-1}\pi_\lambda$ is the required projection. In particular, every \(m\)-dimensional filling supported in an \(m\)-plane minimizes mass among fillings with the same boundary, regardless of the coefficient group.
	
	For comparison, the \emph{inscribed Riemannian volume} $\mu_V^{\mathrm{ir}}$ is induced by the Euclidean norm whose unit ball is the maximal-volume ellipsoid contained in the unit ball of $V$~\cite[Example~4.4]{Iv}. The \emph{circumscribed Riemannian volume} $\mu_V^{\mathrm{cr}}$ is defined using the minimal-volume ellipsoid containing that ball. Monotonicity and John's ellipsoid theorem~\cite{John} give
	\begin{equation}\label{eq:volume-comparison}
		m^{-m/2}\mu_V^{\mathrm{ir}}\le\mu_V^{\mathrm{cr}}
		\le\nu_V\le\mu_V^{\mathrm{ir}}
	\end{equation}
	for every Finsler $m$-volume $\nu$. Thus any two such volumes satisfy
	\[
	m^{-m/2}\nu'_V\le\nu_V\le m^{m/2}\nu'_V.
	\]
	The factor $m^{m/2}$ is optimal: for $V=\ell^m_\infty$, the two ellipsoids are the Euclidean balls of radii $1$ and $\sqrt m$, so $\mu_V^{\mathrm{ir}}=m^{m/2}\mu_V^{\mathrm{cr}}$.
	
	Since every Finsler volume agrees with Riemannian volume on Riemannian fillings, lower bounds obtained with any of these volumes apply to Gromov's problem. Among these choices, the inscribed Riemannian volume gives the strongest lower bounds and is used for this purpose in~\cite{Z2}. Here mass$^\ast$ is useful because its coordinate formula permits explicit mass calculations and calibrations of polyhedral disks.
	
	We conclude with the approximation underlying the finite-dimensional viewpoint in the introduction. The contractions constructed in~\cite[Proposition~A.6]{DP} approximate the identity on compact subsets of $\ell_\infty$ and take values in subspaces isometric to $\ell^n_\infty$. For a compactly supported integral cycle, the homotopy formula gives a connecting chain of arbitrarily small mass between the cycle and its image under such a contraction. Polyhedral approximation in the finite-dimensional space then gives a polyhedral cycle with another small connecting chain; see~\cite[Theorem~4.2(D)]{DP}. For one-dimensional cycles, the same conclusion follows by polygonal approximation of curves. Although the approximation results in \cite{DP} use Hausdorff mass, the connecting chains also have arbitrarily small mass$^\ast$ by~\eqref{eq:volume-comparison}. Adding or subtracting the connecting chain turns an integral filling of one cycle into an integral filling of the other. By injectivity, the polyhedral cycle has the same least integral filling mass in this subspace as in $\ell_\infty$. Thus these filling problems are approximated, with arbitrarily small error in their least integral filling mass, by problems with polyhedral boundary in $\ell^n_\infty$.
	
	\section{Nonuniqueness}\label{sec:nonunique}
	
	The lack of strict convexity in $\ell^n_\infty$ permits infinitely many minimizing geodesics between the same endpoints. A similar nonuniqueness occurs in higher dimensions. Let $Q\subset\ell^m_\infty$ be a cube, and let
	$f_1,\ldots,f_k:Q\to\R$ be $1$-Lipschitz functions vanishing on
	$\partial Q$. Then
	\[
	F:Q\longrightarrow\ell^{m+k}_\infty,\qquad
	F(p)\coloneqq(p,f_1(p),\ldots,f_k(p))
	\]
	is an isometric embedding. The coordinate projection onto the first
	$m$ coordinates shows that $F_\#\curr Q$ and the flat copy of $\curr Q$ minimize mass with the same boundary.
	
	Sums of such minimizing currents can have topologically complicated supports, as the following example shows.
	
	\begin{Expl}\label{ex:cantor}
		Let $K\subset[0,1]$ be a Cantor set containing $0$ and $1$, and put $f(t)\coloneqq\dist(t,K)$. Both curves $\gamma_0(t)\coloneqq(t,0)$ and $\gamma_1(t)\coloneqq(t,f(t))$, $0\le t\le1$, are isometric embeddings into $\ell^2_\infty$. The integral chain
		\[
		T\coloneqq(\gamma_0)_\#\curr{0,1}+(\gamma_1)_\#\curr{0,1}
		\]
		has boundary $2\curr{(1,0)}-2\curr{(0,0)}$ and mass $2$. Projection onto the first coordinate proves minimality. Its two arcs meet precisely on $K\times\{0\}$; over each complementary interval they form a loop; see Figure~\ref{fig:cantor}. With real coefficients, $T/2$ is a minimizer of mass $1$ with endpoint coefficients $-1$ and $1$, and has the same support.
		
	\end{Expl}
	
	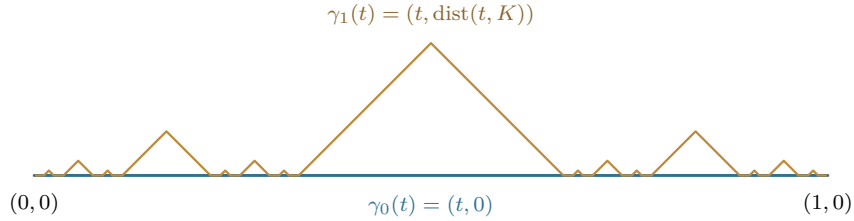
\begin{figure}[htbp]
		\centering
		\begin{tikzpicture}[x=10.5cm,y=10.5cm,line join=round,line cap=round,font=\footnotesize]
			\path[use as bounding box] (-.055,-.075) rectangle (1.055,.205);
			\draw[diskblue,line width=1.1pt] (0,0)--(1,0);
			\draw[conegold!90!black,line width=.85pt] (0.000000000,0.000000000) -- (0.012345679,0.000000000) -- (0.018518519,0.006172840) -- (0.024691358,0.000000000) -- (0.037037037,0.000000000) -- (0.055555556,0.018518519) -- (0.074074074,0.000000000) -- (0.086419753,0.000000000) -- (0.092592593,0.006172840) -- (0.098765432,0.000000000) -- (0.111111111,0.000000000) -- (0.166666667,0.055555556) -- (0.222222222,0.000000000) -- (0.234567901,0.000000000) -- (0.240740741,0.006172840) -- (0.246913580,0.000000000) -- (0.259259259,0.000000000) -- (0.277777778,0.018518519) -- (0.296296296,0.000000000) -- (0.308641975,0.000000000) -- (0.314814815,0.006172840) -- (0.320987654,0.000000000) -- (0.333333333,0.000000000) -- (0.500000000,0.166666667) -- (0.666666667,0.000000000) -- (0.679012346,0.000000000) -- (0.685185185,0.006172840) -- (0.691358025,0.000000000) -- (0.703703704,0.000000000) -- (0.722222222,0.018518519) -- (0.740740741,0.000000000) -- (0.753086420,0.000000000) -- (0.759259259,0.006172840) -- (0.765432099,0.000000000) -- (0.777777778,0.000000000) -- (0.833333333,0.055555556) -- (0.888888889,0.000000000) -- (0.901234568,0.000000000) -- (0.907407407,0.006172840) -- (0.913580247,0.000000000) -- (0.925925926,0.000000000) -- (0.944444444,0.018518519) -- (0.962962963,0.000000000) -- (0.975308642,0.000000000) -- (0.981481481,0.006172840) -- (0.987654321,0.000000000) -- (1.000000000,0.000000000);
			\fill[diskblue] (0,0) circle[radius=.14mm];
			\fill[diskblue] (1,0) circle[radius=.14mm];
			\node[below=3pt] at (0,0) {$(0,0)$};
			\node[below=3pt] at (1,0) {$(1,0)$};
			\node[above=3pt,text=conegold!65!black] at (.5,.166666667) {$\gamma_1(t)=(t,\dist(t,K))$};
			\node[below=4pt,text=diskblue] at (.5,0) {$\gamma_0(t)=(t,0)$};
		\end{tikzpicture}
		\caption{The minimizing chain of Example~\ref{ex:cantor} for the middle-thirds
			Cantor set.}
		\label{fig:cantor}
	\end{figure}
	
	This might suggest that nonuniqueness up to isometry is merely a matter of multiplicity. But even minimizing disks with the same boundary need not be isometric.
	
	\begin{Expl}\label{ex:nonisometric}
		In $\ell^3_\infty$, let
		\[
		D\coloneqq\conv\{0,e_1,e_2\},\qquad c\coloneqq(1/3,1/3,1/3),
		\]
		and let $C$ be the cone built from $\partial D$ to the vertex $c$; see Figure~\ref{fig:nonisometric}. Both are embedded disks with the same triangular boundary. Writing
		\[
		Q\coloneqq\{(x,y):x,y\ge0,\ x+y\le1\},
		\]
		we have
		\[
		D=\{(x,y,0):(x,y)\in Q\},\qquad
		C=\{(x,y,\min\{x,y,1-x-y\}):(x,y)\in Q\}.
		\]
		Orient both disks so that projection onto the $xy$-plane is positive.
		For an oriented triangle $\curr{p_0,p_1,p_2}$, the three entries below
		are the $xy$, $xz$, and $yz$ determinants of the ordered edge vectors
		$p_1-p_0,p_2-p_0$. Its mass is half the largest absolute entry.
		\begin{center}
			\begin{tabular}{lcc}
				\toprule
				Triangle & Coordinate determinants & Mass\\
				\midrule
				$\curr{0,e_1,e_2}$ & $(1,0,0)$ & $1/2$\\
				$\curr{0,e_1,c}$ & $(1/3,1/3,0)$ & $1/6$\\
				$\curr{0,c,e_2}$ & $(1/3,0,-1/3)$ & $1/6$\\
				$\curr{c,e_1,e_2}$ & $(1/3,-1/3,1/3)$ & $1/6$\\
				\bottomrule
			\end{tabular}
		\end{center}
		The $xy$ determinant attains the maximum on every face. Consequently
		$dx\wedge dy$ calibrates both disks, and
		\[
		\partial\curr C=\partial\curr D,\qquad
		\bM(\curr C)=\bM(\curr D)=\frac12.
		\]
		
		Nevertheless, the two disks are not isometric. The tangent plane of $D$ is $\{z=0\}$, whose norm unit disk is a square. In the relative interior of the face $\conv\{c,e_1,e_2\}$, the tangent plane is $\{x+y+z=0\}$, whose norm unit disk is a hexagon. A linear isometry would map one unit disk onto the other, which is impossible because their numbers of extreme points differ. An isometry of the disks would, at an interior differentiability point of this face, induce just such a linear isometry. This applies both to the restricted ambient metrics and to the intrinsic length metrics of the disks.
		
	\end{Expl}
	
	\begin{figure}[htbp]
		\centering
		\tdplotsetmaincoords{72}{95}
		\begin{tikzpicture}[tdplot_main_coords,scale=4.3,line join=round,line cap=round]
			\draw[->,draw=black!60,line width=.45pt] (0,0,0) -- (1.23,0,0);
			\node[font=\small,inner sep=2pt,anchor=north east,xshift=3pt] at (1.23,0,0) {$x$};
			\draw[->,draw=black!60,line width=.45pt] (0,0,0) -- (0,1.1,0);
			\node[font=\small,inner sep=2pt,anchor=west] at (0,1.1,0) {$y$};
			\draw[->,draw=black!60,line width=.45pt] (0,0,0) -- (0,0,0.4);
			\node[font=\small,inner sep=2pt,anchor=south] at (0,0,0.4) {$z$};
			\path[fill=diskblue,fill opacity=.55,draw=diskblue,line width=.7pt] (0,0,0) -- (1,0,0) -- (0,1,0) -- cycle;
			\path[fill=conegold,fill opacity=.12,draw=conegold!70!black,line width=.65pt] (0.3333333,0.3333333,0.3333333) -- (0,1,0) -- (0,0,0) -- cycle;
			\path[fill=conegold,fill opacity=.12,draw=conegold!70!black,line width=.65pt] (0.3333333,0.3333333,0.3333333) -- (0,0,0) -- (1,0,0) -- cycle;
			\path[fill=conegold,fill opacity=.12,draw=conegold!70!black,line width=.65pt] (0.3333333,0.3333333,0.3333333) -- (1,0,0) -- (0,1,0) -- cycle;
			\draw[boundaryedge] (0,0,0) -- (1,0,0);
			\draw[boundaryedge] (1,0,0) -- (0,1,0);
			\draw[boundaryedge] (0,1,0) -- (0,0,0);
			\filldraw[fill=black!85,draw=white,line width=.3pt] (0,0,0) circle (.32pt);
			\node[text=black!85,font=\footnotesize,inner sep=2.5pt,anchor=east] at (0,0,0) {$0$};
			\filldraw[fill=black!85,draw=white,line width=.3pt] (1,0,0) circle (.32pt);
			\node[text=black!85,font=\footnotesize,inner sep=2.5pt,anchor=east] at (1,0,0) {$e_1$};
			\filldraw[fill=black!85,draw=white,line width=.3pt] (0,1,0) circle (.32pt);
			\node[text=black!85,font=\footnotesize,inner sep=4.0pt,anchor=north] at (0,1,0) {$e_2$};
			\filldraw[fill=conegold!65!black,draw=white,line width=.3pt] (0.3333333,0.3333333,0.3333333) circle (.32pt);
			\node[text=conegold!65!black,font=\footnotesize,inner sep=2.5pt,anchor=south] at (0.3333333,0.3333333,0.3333333) {$c$};
		\end{tikzpicture}
		\caption{Two nonisometric minimizing disks with boundary $\partial\conv\{0,e_1,e_2\}$.}
		\label{fig:nonisometric}
	\end{figure}
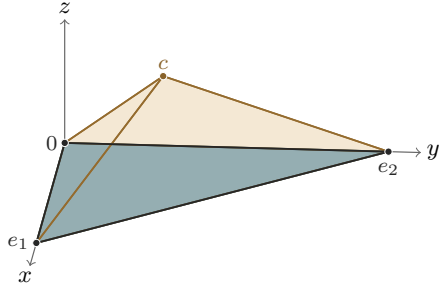
	
	\section{Polygonal approximations of the circle}\label{sec:discrete}
	
	Let $\mathbb S^1\coloneqq\R/(2\pi\Z)$ have its intrinsic distance $d$
	and the orientation induced by the standard orientation of $\R$. A centered model for its injective hull is
	\begin{equation}\label{eq:circle-hull}
		E(\mathbb S^1)\coloneqq\{f:\R\to\R:\Lip(f)\le1,\ f(t+\pi)=-f(t)\},
	\end{equation}
	with the supremum metric; see~\cite[Proposition~2.7]{Z2}. The unique isometric copy of \(\mathbb S^1\) in \(E(\mathbb S^1)\) consists of the functions \(f_\theta(t)\coloneqq \pi/2-d(\theta,t)\), indexed by \(\theta\in\mathbb S^1\). 
	
	Fix $n\ge2$, put $h_n\coloneqq\pi/n$, and set $\alpha_i\coloneqq ih_n\in\mathbb S^1$ for $1\le i\le2n$. Define $\rho_n : \mathbb S^1 \to \ell^n_\infty$ by
	\[
	\rho_n(\theta) \coloneqq \bigl(\pi/2 - d(\theta,\alpha_i)\bigr)_{i=1}^n
	\]
	and consider
	\[
	S_n\coloneqq \rho_n(\mathbb S^1),\qquad
	\Gamma_n\coloneqq (\rho_n)_\#\curr{\mathbb S^1}.
	\]
	We write $\Fill(S_n)\coloneqq\Fill(\Gamma_n)$. Each coordinate function of $\rho_n$ is $1$-Lipschitz. Distances to the two non-antipodal points $\alpha_1,\alpha_2$ already distinguish points of the circle, so $\rho_n$ is injective.
	The map is affine between successive $\alpha_i$ and has speed one in the maximum norm.
	
	We extend coordinates by $x_{i+n}\coloneqq-x_i$ and set
	\[
	\sigma_n(x_1,\ldots,x_n) \coloneqq (-x_n,x_1,\ldots,x_{n-1}).
	\]
	Write $v_i\coloneqq\rho_n(\alpha_i)$ for $1\le i\le2n$.
	The vertices of $S_n$ form the orbit under $\sigma_n$ of
	\[
	v_1
	= h_n\bigl(n/2,n/2-1,\ldots,1-n/2\bigr).
	\]
	They satisfy $v_{i+n}=-v_i$ and $\sigma_nv_i=v_{i+1}$,
	with indices understood modulo $2n$.
	With $v_{2n+1}\coloneqq v_1$ and the orientation induced by
	increasing $\theta$,
	\begin{equation}\label{eq:orientation}
		\Gamma_n=\sum_{i=1}^{2n}\curr{v_i,v_{i+1}}.
	\end{equation}
	For $1\le k<n$, let $R_k$ be the rectangle
	\begin{equation}\label{eq:projection-rectangle}
		R_k\coloneqq\{(s,t):|s-t|\le k,\ |s+t|\le n-k\}.
	\end{equation}
	Its consecutive vertices are
	\[
	A_k\coloneqq(n/2,n/2-k),\quad B_k\coloneqq(n/2-k,n/2),\quad -A_k,\quad -B_k,
	\]
	and $\cL^2(R_k)=2k(n-k)$. Give $R_k$ the standard orientation of $\R^2$. For $1\le i<j\le n$, the projection of $S_n$ traverses $h_nA_{j-i},h_nB_{j-i},-h_nA_{j-i},-h_nB_{j-i}$ in this counterclockwise order, starting at $\alpha_i$. Thus
	\begin{equation}\label{eq:projected-boundary}
		(\pi_{ij})_\#\Gamma_n=\partial\curr{h_nR_{j-i}}.
	\end{equation}
	
	\subsection{Injective hull}
	\label{sec:injective}
	
	The following is the finite analogue of \eqref{eq:circle-hull}.
	We include the proof for convenience.
	
	\begin{Prop}\label{prop:hull}
		The injective hulls of $S_n\subset\ell^n_\infty$ and of its vertex set are both isometrically realized as
		\begin{equation}\label{eq:finite-hull}
			E(S_n)=h_nE_n,
			\quad\text{where}\quad
			E_n\coloneqq\{x\in\R^n:|x_i-x_{i+1}|\le1
			\text{ for }1\le i\le n\},
		\end{equation}
		with $x_{n+1}=-x_1$. In particular, $h_nE_n$ is a
		$1$-Lipschitz retract of $\ell^n_\infty$.
	\end{Prop}
	
	\begin{proof}
		Let $C\coloneqq\{1,\ldots,2n\}$ be the vertex set of a cycle with
		unit edge lengths, labelled in cyclic order. Its metric is given by
		\[
		d(p,q)=\min\{|p-q|,2n-|p-q|\}.
		\]
		Write $\bar p\coloneqq p+n$, with indices understood modulo $2n$.
		We first identify $E(C)$ with $E_n$. Recall from~\cite[Section~3]{L2} that $E(C)$, equipped with the
		supremum metric, consists of functions $g:C\to\R$ satisfying
		\[
		g(p)+g(q)\ge d(p,q)\qquad\text{for all }p,q\in C,
		\]
		where, for every $p\in C$, equality holds for at least one $q\in C$.
		Such functions are $1$-Lipschitz, since
		\[
		g(p)=\max_{q\in C}\{d(p,q)-g(q)\}.
		\]
		Fix $p$ and choose $q$ attaining equality. Using
		$d(p,q)+d(q,\bar p)=n$, we obtain
		\[
		n\le g(p)+g(\bar p)
		\le g(p)+g(q)+d(q,\bar p)=n.
		\]
		Conversely, if $g:C\to\R$ is $1$-Lipschitz and satisfies
		$g(p)+g(\bar p)=n$ for every $p$, then
		\[
		g(p)+g(q)
		=n+g(p)-g(\bar q)
		\ge n-d(p,\bar q)=d(p,q),
		\]
		with equality when $q=\bar p$. Thus
		\[
		E(C)=
		\{g:C\to\R:\Lip(g)\le1,\ g(p)+g(\bar p)=n
		\text{ for all }p\in C\}.
		\]
		
		For $g\in E(C)$, set $x_p\coloneqq n/2 - g(p)$ for $1\le p\le2n$
		and write $x\coloneqq(x_1,\ldots,x_n)$.
		The antipodal identity becomes $x_{p+n}=-x_p$.
		Moreover, the Lipschitz condition is equivalent to the bounds
		between neighbouring indices: all other bounds follow by summing
		along shortest paths in the cycle. By antiperiodicity, these bounds
		reduce to $|x_p-x_{p+1}|\le1$ for all $1 \le p \le n$, where $x_{n+1} = -x_1$.
		
		Consequently, $g\mapsto x$ is a bijection from $E(C)$ onto $E_n$. It is an isometry, since for any $g,g'\in E(C)$ and their
		corresponding vectors $x,x'$,
		\[
		\|g-g'\|_\infty
		=\max_{1\le p\le n}|g(p)-g'(p)|
		=\|x-x'\|_\infty.
		\]
		
		The standard embedding $C\to E(C)$ sends $i$ to the distance function $g_i(j)\coloneqq d(i,j)$. Under the preceding
		identification and scaling by $h_n$, this function corresponds to
		\[
		h_n\bigl(n/2-d(i,1),\ldots,n/2-d(i,n)\bigr)
		=\rho_n(\alpha_i)=v_i.
		\]
		Hence $h_nE_n$ is the injective hull of the vertex set of $S_n$. The defining inequalities of $E_n$ show that $h_nE_n$ is convex.
		Since it contains all vertices of the polygon $S_n$, it also
		contains every edge, and hence $S_n$ itself.
		No proper injective subspace of $h_nE_n$ contains the vertex set of $S_n$, so $h_nE_n$ is also the injective hull of $S_n$. Finally, injectivity extends its identity map to a $1$-Lipschitz retraction of $\ell^n_\infty$ onto $h_nE_n$.
	\end{proof}
	
	For $1\le i\le n$ and $1\le k<n$, the projection of $E_n$ onto
	$(x_i,x_{i+k})$ is exactly $R_k$. The defining inequalities of $R_k$
	follow by summing the neighbouring bounds along the paths from $i$ to
	$i+k$ and from $i+k$ to $i+n$ in the $2n$-cycle. Conversely, the
	projection is convex and contains all four vertices of $R_k$, which
	are projections of vertices of $S_n/h_n$. The same neighbouring bounds give $|x_i|\le n/2$ for every $x\in E_n$.
	
	The retraction in Proposition~\ref{prop:hull} sends any filling of $\Gamma_n$ into $h_nE_n$ without increasing its mass. Thus all subsequent minimization problems may be restricted to this compact set.
	
	\subsection{Calibrations}\label{subsec:calibrations}
	
	The following projection estimate gives the calibrations used below.
	
	\begin{Prop}
		\label{prop:calibration}
		Let $E\subset\ell^n_\infty$ be compact, let $2\le m\le n$, and let
		$S\in\bN_{m-1}(\ell^n_\infty)$ be a cycle supported in $E$.
		For each $\lambda\in\Lambda(n,m)$, let
		$\theta_\lambda\in L^1_c(\R^m)$ be the unique function satisfying
		$\partial\curr{\theta_\lambda}=(\pi_\lambda)_\#S$.
		Suppose the bounded Borel functions $f_\lambda:\R^m\to[0,\infty)$ satisfy
		\begin{equation}\label{eq:calibration_condition}
			\sum_{\lambda\in\Lambda(n,m)}f_\lambda(\pi_\lambda(x))\le1
			\qquad\text{for all }x\in E.
		\end{equation}
		Then every $T\in\bN_m(\ell^n_\infty)$ supported in $E$ with $\partial T=S$ satisfies
		\begin{equation}\label{eq:projection-bound}
			\bM(T)\ge\sum_{\lambda\in\Lambda(n,m)}
			\int_{\R^m}f_\lambda(z)|\theta_\lambda(z)|\,d\cL^m(z),
		\end{equation}
		and any such filling attaining equality is mass minimizing among currents supported in $E$.
	\end{Prop}
	
	\begin{proof}
		The projected cycle $(\pi_\lambda)_\#S$ has a compactly supported normal
		cone filling in $\R^m$, which is of the form $\curr{\theta_\lambda}$
		by~\cite[Theorem~3.7]{AK}. Uniqueness follows from the constancy
		theorem~\cite[4.1.7]{F}. Consequently
		$(\pi_\lambda)_\#T=\curr{\theta_\lambda}$ for every filling under consideration.
		Since $\pi_\lambda$ is $1$-Lipschitz into $\ell^m_\infty$, whose
		mass$^\ast$ agrees with $\cL^m$, \eqref{eq:local-pushforward} gives
		\[
		|\theta_\lambda|\cL^m=\|\curr{\theta_\lambda}\|
		\le(\pi_\lambda)_\#\|T\|.
		\]
		Integrating $f_\lambda$ and summing over $\lambda$ now gives
		\eqref{eq:projection-bound}, since the coefficient condition holds on
		$\spt T\subset E$.
	\end{proof}
	
	The mass$^\ast$ definition of volume enters only through the fact that it agrees with Lebesgue measure on $\ell^m_\infty$. For other definitions of volume, the lower bound may require a correction factor.
	
	The weights in Proposition~\ref{prop:calibration} also define a closed Whitney flat form
	\[
	\omega\coloneqq\sum_{\lambda\in\Lambda(n,m)}
	\bigl((f_\lambda\operatorname{sgn}\theta_\lambda)\circ\pi_\lambda\bigr)\,dx_\lambda.
	\]
	Each summand is distributionally closed because its coefficient depends
	only on the coordinates in its wedge; see~\cite[Chapter~IX]{Wh}
	or~\cite[4.1.19]{F}. Its action on a filling $T$ as above is
	\begin{equation}\label{eq:calibration-bound}
		T(\omega)=\sum_{\lambda\in\Lambda(n,m)}
		\int_{\R^m}f_\lambda|\theta_\lambda|\,d\cL^m\le\bM(T).
	\end{equation}
	For the positive projection multiplicities below, the sign factor may be omitted.
	
	\begin{Rem}
		\label{rem:coefficients}
		The proposition also holds for rectifiable $G$-chains in the sense of De Pauw and Hardt~\cite{DPH}, where $G$ is any complete normed Abelian group. Here $\theta_\lambda$ is $G$-valued and $|\theta_\lambda|$ is its norm. We replace the Hausdorff volume used in~\cite{DPH} by mass$^\ast$: on rectifiable chains, mass integrates the norm of the multiplicity against this volume, as in~\cite[Definition~2.4]{Z2}. The proof uses the same push-forward estimate and the top-dimensional constancy theorem~\cite[Theorem~6.4]{DPH2}. No additional assumption on $G$ is needed for these lower bounds.
	\end{Rem}
	
	\subsection{The case \texorpdfstring{$n=2$}{n=2}}
	
	The curve $S_2$ is the boundary of the square $D_2$ with vertices
	$(\pm\pi/2,0)$ and $(0,\pm\pi/2)$. This square is also $E(S_2)$. The
	coordinate form $dx_1\wedge dx_2$ calibrates its filling, and
	\[
	\Fill(S_2)=\frac{\pi^2}{2}=4.934802\ldots.
	\]
	
	\subsection{The case \texorpdfstring{$n=3$}{n=3}}\label{subsec:n3}
	
	Put $h\coloneqq h_3=\pi/3$ and set
	\[
	v_1=\frac h2(3,1,-1),\qquad
	w\coloneqq\frac h2(2,0,-2),\qquad w'\coloneqq\frac h2(2,2,0).
	\]
	Here $w$ and $w'$ are the midpoints of the two edges incident to $v_1$.
	The six images of the oriented triangles
	\[
	A\coloneqq\curr{v_1,w',w},\qquad B\coloneqq\curr{0,w,w'}
	\]
	under $\sigma_3$ form a filling $D_3$ of $\Gamma_3$. The images of $B$ tile a
	flat central hexagon in the plane $x_1-x_2+x_3=0$; the images of $A$ are the six
	outer triangles. Orthogonal projection onto this plane maps these faces to a
	subdivision of a hexagon, so their union is an embedded disk. Writing
	$e_1,e_2,e_3$ for the standard basis, we have
	\[
	\begin{aligned}
		(w'-v_1)\wedge(w-v_1)
		&=\frac{h^2}{2}(e_1\wedge e_2+e_1\wedge e_3),\\
		w\wedge w'
		&=h^2(e_1\wedge e_2+e_1\wedge e_3+e_2\wedge e_3).
	\end{aligned}
	\]
	By~\eqref{eq:simplex-mass}, the mass of each triangle is half the largest absolute coefficient, giving
	\[
	\bM(A)=\frac{h^2}{4},\qquad \bM(B)=\frac{h^2}{2},\qquad
	\bM(D_3)=6\left(\frac{h^2}{4}+\frac{h^2}{2}\right)
	=\frac{\pi^2}{2}.
	\]
	
	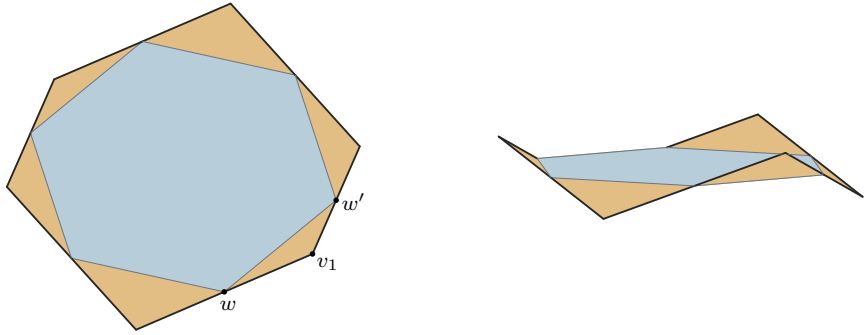
\begin{figure}[htbp]
		\centering
		\begin{minipage}{.48\textwidth}\centering
			\begin{tikzpicture}[scale=1.48,line join=round,line cap=round]
				\path[use as bounding box] (-1.72735,-1.607667) rectangle (1.72735,1.607667);
				\path[fill=conegold!65,draw=conegold!65,line width=.16pt] (1.154701,-0.7802981) -- (1.366025,-0.2998305) -- (0.3660254,-1.118983) -- cycle;
				\path[fill=conegold!65,draw=conegold!65,line width=.16pt] (-0.4226497,-1.457667) -- (0.3660254,-1.118983) -- (-1,-0.819152) -- cycle;
				\path[fill=conegold!65,draw=conegold!65,line width=.16pt] (-1.57735,-0.1806371) -- (-1,-0.819152) -- (-1.366025,0.2998305) -- cycle;
				\path[fill=conegold!65,draw=conegold!65,line width=.16pt] (-1.154701,0.7802981) -- (-1.366025,0.2998305) -- (-0.3660254,1.118983) -- cycle;
				\path[fill=conegold!65,draw=conegold!65,line width=.16pt] (0.4226497,1.457667) -- (-0.3660254,1.118983) -- (1,0.819152) -- cycle;
				\path[fill=conegold!65,draw=conegold!65,line width=.16pt] (1.57735,0.1806371) -- (1,0.819152) -- (1.366025,-0.2998305) -- cycle;
				\path[fill=diskblue!35,draw=diskblue!35,line width=.16pt] (0.3660254,-1.118983) -- (1.366025,-0.2998305) -- (1,0.819152) -- (-0.3660254,1.118983) -- (-1.366025,0.2998305) -- (-1,-0.819152) -- cycle;
				\draw[boundaryedge] (1.366025,-0.2998305) -- (1.154701,-0.7802981);
				\draw[surfaceedge] (0.3660254,-1.118983) -- (1.366025,-0.2998305);
				\draw[boundaryedge] (0.3660254,-1.118983) -- (1.154701,-0.7802981);
				\draw[boundaryedge] (-0.4226497,-1.457667) -- (0.3660254,-1.118983);
				\draw[surfaceedge] (-1,-0.819152) -- (0.3660254,-1.118983);
				\draw[boundaryedge] (-1,-0.819152) -- (-0.4226497,-1.457667);
				\draw[boundaryedge] (-1.57735,-0.1806371) -- (-1,-0.819152);
				\draw[surfaceedge] (-1.366025,0.2998305) -- (-1,-0.819152);
				\draw[boundaryedge] (-1.366025,0.2998305) -- (-1.57735,-0.1806371);
				\draw[boundaryedge] (-1.154701,0.7802981) -- (-1.366025,0.2998305);
				\draw[surfaceedge] (-1.366025,0.2998305) -- (-0.3660254,1.118983);
				\draw[boundaryedge] (-1.154701,0.7802981) -- (-0.3660254,1.118983);
				\draw[boundaryedge] (-0.3660254,1.118983) -- (0.4226497,1.457667);
				\draw[surfaceedge] (-0.3660254,1.118983) -- (1,0.819152);
				\draw[boundaryedge] (0.4226497,1.457667) -- (1,0.819152);
				\draw[boundaryedge] (1,0.819152) -- (1.57735,0.1806371);
				\draw[surfaceedge] (1,0.819152) -- (1.366025,-0.2998305);
				\draw[boundaryedge] (1.57735,0.1806371) -- (1.366025,-0.2998305);
				\fill (1.154701,-0.7802981) circle (.7pt);
				\node[font=\footnotesize,inner sep=1.4pt,anchor=north west] at (1.154701,-0.7802981) {$v_1$};
				\fill (0.3660254,-1.118983) circle (.7pt);
				\node[font=\footnotesize,inner sep=3.1pt,anchor=north,xshift=1.5pt] at (0.3660254,-1.118983) {$w$};
				\fill (1.366025,-0.2998305) circle (.7pt);
				\node[font=\footnotesize,inner sep=1.9pt,anchor=west,yshift=-0.5pt] at (1.366025,-0.2998305) {$w'$};
			\end{tikzpicture}
		\end{minipage}\hfill
		\begin{minipage}{.48\textwidth}\centering
			\begin{tikzpicture}[scale=1.48,line join=round,line cap=round]
				\path[use as bounding box] (-1.776779,-0.6168893) rectangle (1.776779,0.6168893);
				\path[fill=conegold!65,draw=conegold!65,line width=.16pt] (0.1232568,-0.1716934) -- (1.281713,-0.07283795) -- (0.9366464,0.1235025) -- cycle;
				\path[fill=conegold!65,draw=conegold!65,line width=.16pt] (-0.6901327,-0.4668893) -- (0.1232568,-0.1716934) -- (-1.158456,-0.09885548) -- cycle;
				\path[fill=conegold!65,draw=conegold!65,line width=.16pt] (-1.626779,0.2691784) -- (-1.158456,-0.09885548) -- (-1.281713,0.07283795) -- cycle;
				\path[fill=conegold!65,draw=conegold!65,line width=.16pt] (-0.1232568,0.1716934) -- (-1.281713,0.07283795) -- (-1.281713,0.07283795) -- (-0.1232568,0.1716934) -- cycle;
				\path[fill=conegold!65,draw=conegold!65,line width=.16pt] (0.6901327,0.4668893) -- (0.4721702,0.3877862) -- (0.9601581,0.1101245) -- (1.158456,0.09885548) -- cycle;
				\path[fill=conegold!65,draw=conegold!65,line width=.16pt] (0.4721702,0.3877862) -- (-0.1232568,0.1716934) -- (0.9079582,0.1130909) -- (0.9366464,0.1235025) -- cycle;
				\path[fill=conegold!65,draw=conegold!65,line width=.16pt] (1.626779,-0.2691784) -- (1.158456,0.09885548) -- (1.281713,-0.07283795) -- cycle;
				\path[fill=diskblue!35,draw=diskblue!35,line width=.16pt] (1.281713,-0.07283795) -- (1.158456,0.09885548) -- (0.9601581,0.1101245) -- cycle;
				\path[fill=diskblue!35,draw=diskblue!35,line width=.16pt] (0.1232568,-0.1716934) -- (0.9079582,0.1130909) -- (-0.1232568,0.1716934) -- (-1.281713,0.07283795) -- (-1.158456,-0.09885548) -- cycle;
				\path[fill=diskblue!35,draw=diskblue!35,line width=.16pt] (0.1232568,-0.1716934) -- (1.281713,-0.07283795) -- (1.281713,-0.07283795) -- (0.1232568,-0.1716934) -- cycle;
				\draw[boundaryedge] (1.281713,-0.07283795) -- (0.9366464,0.1235025);
				\draw[surfaceedge] (0.1232568,-0.1716934) -- (1.281713,-0.07283795);
				\draw[boundaryedge] (0.1232568,-0.1716934) -- (0.9366464,0.1235025);
				\draw[boundaryedge] (-0.6901327,-0.4668893) -- (0.1232568,-0.1716934);
				\draw[surfaceedge] (-1.158456,-0.09885548) -- (0.1232568,-0.1716934);
				\draw[boundaryedge] (-1.158456,-0.09885548) -- (-0.6901327,-0.4668893);
				\draw[boundaryedge] (-1.626779,0.2691784) -- (-1.158456,-0.09885548);
				\draw[surfaceedge] (-1.281713,0.07283795) -- (-1.158456,-0.09885548);
				\draw[boundaryedge] (-1.281713,0.07283795) -- (-1.626779,0.2691784);
				\draw[surfaceedge] (-1.281713,0.07283795) -- (-0.1232568,0.1716934);
				\draw[boundaryedge] (-0.1232568,0.1716934) -- (0.6901327,0.4668893);
				\draw[surfaceedge] (-0.1232568,0.1716934) -- (0.9079582,0.1130909);
				\draw[surfaceedge] (0.9601581,0.1101245) -- (1.158456,0.09885548);
				\draw[boundaryedge] (0.6901327,0.4668893) -- (1.158456,0.09885548);
				\draw[boundaryedge] (1.158456,0.09885548) -- (1.626779,-0.2691784);
				\draw[surfaceedge] (1.158456,0.09885548) -- (1.281713,-0.07283795);
				\draw[boundaryedge] (1.626779,-0.2691784) -- (1.281713,-0.07283795);
			\end{tikzpicture}
		\end{minipage}
		\caption{The disk $D_3$. The central hexagon is a single flat region and the six triangles
			fold alternately above and below it.}
		\label{fig:n3}
	\end{figure}
	
	This shows that $\Fill(S_3)\le\pi^2/2$. For the reverse inequality,
	define $f(s,t)\coloneqq\tfrac12\bone_{\{st>0\}}$ and use the
	antiperiodic convention $x_4\coloneqq-x_1$. The form
	\begin{equation}\label{eq:n3-form}
		\omega\coloneqq\sum_{i=1}^3 f(x_i,x_{i+1})\,dx_i\wedge dx_{i+1}
	\end{equation}
	satisfies the coefficient bound of Proposition~\ref{prop:calibration} on all of $\R^3$: the three products
	\[
	x_1x_2,\quad x_2x_3,\quad x_3x_4 = x_3(-x_1)
	\]
	cannot all be positive. In ordinary coordinates its
	$dx_1\wedge dx_3$ coefficient is $\tfrac12\bone_{\{x_1x_3<0\}}$.
	
	Each signed coordinate projection $x\mapsto(x_i,x_{i+1})$, $i=1,2,3$,
	sends $\Gamma_3$ to $\partial\curr{hR_1}$. The rectangle $R_1$ has area $4$, and its intersection with each of the first and third quadrants has area $3/2$. Hence
	\[
	\int_{hR_1}f\,d\cL^2
	=\frac12\left(\frac32+\frac32\right)h^2
	=\frac32h^2.
	\]
	Proposition~\ref{prop:calibration} now gives $\bM(T)\ge\frac92h^2=\pi^2/2$
	for every competitor. Thus $D_3$ is absolutely minimizing.
	
	\subsection{The case \texorpdfstring{$n=4$}{n=4}}\label{subsec:n4}
	
	Put $h\coloneqq h_4=\pi/4$. In the original scale of $S_4$,
	the points used in the construction have coordinates
	\begin{align*}
		v_1&=\tfrac h5(10,5,0,-5),
		&v_2&=\tfrac h5(5,10,5,0),\\
		v_3&=\tfrac h5(0,5,10,5),&&\\[2pt]
		a_1&\coloneqq\tfrac h5(7,8,3,-2),
		&a_2&\coloneqq\tfrac h5(3,8,7,2),\\
		a_3&\coloneqq\tfrac h5(2,7,8,3),&&\\[2pt]
		b_1&\coloneqq\tfrac h5(6,6,2,-2),
		&b_2&\coloneqq\tfrac h5(2,6,6,2),\\
		c_1&\coloneqq\tfrac h5(4,6,4,0),
		&c_2&\coloneqq\tfrac h5(0,4,6,4).
	\end{align*}
	They satisfy
	\begin{gather*}
		v_2=\sigma_4v_1,\qquad v_3=\sigma_4v_2,\\
		a_1=\tfrac25v_1+\tfrac35v_2,\qquad
		a_2=\tfrac35v_2+\tfrac25v_3,\qquad
		a_3=\sigma_4a_1,\\
		b_2=\sigma_4b_1,\qquad
		c_1=\tfrac12(b_1+b_2),\qquad
		c_2=\sigma_4c_1.
	\end{gather*}
	In particular, $a_1,v_2,a_2,a_3$ occur in this order along
	$\Gamma_4$, and $c_1$ is the midpoint of $[b_1,b_2]$.
	
	The points $a_1,a_2,b_2,b_1$, in that order, form an oriented
	parallelogram $P$. Let $D_4$ be the sum of the eight images
	under $\sigma_4$ of the five oriented faces
	\begin{equation}\label{eq:n4-faces}
		P,\quad
		\curr{a_1,v_2,a_2},\quad
		\curr{a_2,a_3,b_2},\quad
		\curr{b_2,c_2,c_1},\quad
		\curr{0,c_1,c_2}.
	\end{equation}
	The common edges cancel after subdividing $[b_1,b_2]$ at $c_1$.
	The remaining edges subdivide $\Gamma_4$. All vertices belong
	to $hE_4$, hence all faces do as well.
	Figure~\ref{fig:n4-block} shows the five-face building block.
	
	\begin{figure}[htbp]
		\centering
		\begin{tikzpicture}[
			scale=1.5,
			line join=round,
			line cap=round,
			font=\footnotesize
			]
			\coordinate (O) at (0,0);
			\coordinate (v2) at (1.707107,1.707107);
			\coordinate (a1) at (1.989949,1.024264);
			\coordinate (a2) at (1.024264,1.989949);
			\coordinate (a3) at (.6828427,2.131371);
			\coordinate (b1) at (1.648528,.6828427);
			\coordinate (b2) at (.6828427,1.648528);
			\coordinate (c1) at (1.165685,1.165685);
			\coordinate (c2) at (0,1.648528);
			
			\path[fill=rectcoral!75]
			(a1)--(a2)--(b2)--(c1)--(b1)--cycle;
			\path[fill=conegold!75]
			(a1)--(v2)--(a2)--cycle;
			\path[fill=tipcream]
			(a2)--(a3)--(b2)--cycle;
			\path[fill=collarteal!65]
			(b2)--(c2)--(c1)--cycle;
			\path[fill=diskblue!35]
			(O)--(c1)--(c2)--cycle;
			
			\draw[surfaceedge]
			(a1)--(b1)--(c1)--(O)--(c2)--(b2)--(a3);
			\draw[surfaceedge] (a2)--(b2)--(c1);
			\draw[surfaceedge] (a1)--(a2);
			\draw[surfaceedge] (c1)--(c2);
			\draw[boundaryedge] (a1)--(v2)--(a2)--(a3);
			
			\foreach \p in {O,v2,a1,a2,a3,b1,b2,c1,c2}
			\fill[black!75] (\p) circle[radius=.7pt];
			
			\node[anchor=north,inner sep=3pt,xshift=-2pt]
			at (O) {$0$};
			\node[anchor=south west,inner sep=.75pt]
			at (v2) {$v_2$};
			\node[anchor=west,inner sep=2.5pt,yshift=-2pt]
			at (a1) {$a_1$};
			\node[anchor=south,inner sep=2pt,xshift=2pt]
			at (a2) {$a_2$};
			\node[anchor=south east,inner sep=1pt]
			at (a3) {$a_3$};
			\node[anchor=north,inner sep=2pt]
			at (b1) {$b_1$};
			\node[anchor=south east,inner sep=1.5pt]
			at (b2) {$b_2$};
			\node[anchor=north,inner sep=3.5pt,yshift=-1.5pt]
			at (c1) {$c_1$};
			\node[anchor=east,inner sep=1.5pt,yshift=1.5pt]
			at (c2) {$c_2$};
			\node at (1.34,1.34) {$P$};
		\end{tikzpicture}
		\caption{The five faces in~\eqref{eq:n4-faces}, shown in
			the same projection and colors as the left panel of
			Figure~\ref{fig:n4}.}
		\label{fig:n4-block}
	\end{figure}
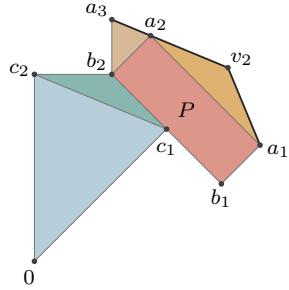
	
	The following table gives the signed areas of the six coordinate projections
	in the order $12,13,14,23,24,34$, together with the mass of each face.
	Areas and masses are expressed in units of $(h/5)^2$. The mass is the
	largest absolute projection area.
	\begin{center}
		\begin{tabular}{lcc}
			\toprule
			Face & Signed projection areas & Mass\\
			\midrule
			$P$ & $(8,8,4,8,8,4)$ & $8$\\
			$\curr{a_1,v_2,a_2}$ & $(4,0,0,4,4,0)$ & $4$\\
			$\curr{a_2,a_3,b_2}$ & $(\tfrac12,1,\tfrac12,\tfrac32,1,\tfrac12)$ & $3/2$\\
			$\curr{b_2,c_2,c_1}$ & $(2,2,0,2,2,2)$ & $2$\\
			$\curr{0,c_1,c_2}$ & $(8,12,8,10,12,8)$ & $12$\\
			\bottomrule
		\end{tabular}
	\end{center}
	Consequently,
	\begin{equation}\label{eq:n4-upper}
		\bM(D_4)=8\left(8+4+\frac32+2+12\right)\left(\frac h5\right)^2
		=\frac{44}{5}h^2=\frac{11}{20}\pi^2.
	\end{equation}
	
	The projections in Figure~\ref{fig:n4} are orthogonal with respect to the standard Euclidean inner product on $\R^4$. We use the orthonormal basis
	\begin{align*}
		a&\coloneqq\tfrac12(\sqrt2,1,0,-1),&
		b&\coloneqq\tfrac12(0,1,\sqrt2,1),\\
		c&\coloneqq\tfrac12(1,-\sqrt2,1,0),&
		d&\coloneqq\tfrac12(-1,0,1,-\sqrt2).
	\end{align*}
	Projection onto $\operatorname{span}(a,b)$ identifies $D_4$ with a subdivided regular octagon; hence $D_4$ is an embedded disk. Figure~\ref{fig:coordinate-projections} shows the two coordinate projections used by the calibration constructed below.
	
	\begin{figure}[htbp]
		\centering
		\begin{minipage}[t]{.37\textwidth}\centering
			\parbox[c][4.36cm][c]{\linewidth}{\centering
				\begin{tikzpicture}[scale=0.85,line join=round,line cap=round]
					\path[use as bounding box] (-2.564214,-2.564214) rectangle (2.564214,2.564214);
					\path[fill=rectcoral!75,draw=rectcoral!75,line width=.16pt] (1.024264,1.989949) -- (0.6828427,1.648528) -- (1.165685,1.165685) -- (1.648528,0.6828427) -- (1.989949,1.024264) -- cycle;
					\path[fill=conegold!75,draw=conegold!75,line width=.16pt] (1.707107,1.707107) -- (1.024264,1.989949) -- (1.989949,1.024264) -- cycle;
					\path[fill=tipcream,draw=tipcream,line width=.16pt] (0.6828427,2.131371) -- (0.6828427,1.648528) -- (1.024264,1.989949) -- cycle;
					\path[fill=collarteal!65,draw=collarteal!65,line width=.16pt] (0.6828427,1.648528) -- (0,1.648528) -- (1.165685,1.165685) -- cycle;
					\path[fill=diskblue!35,draw=diskblue!35,line width=.16pt] (0,1.648528) -- (0,0) -- (1.165685,1.165685) -- cycle;
					\path[fill=rectcoral!75,draw=rectcoral!75,line width=.16pt] (2.131371,0.6828427) -- (1.648528,0.6828427) -- (1.648528,0) -- (1.648528,-0.6828427) -- (2.131371,-0.6828427) -- cycle;
					\path[fill=conegold!75,draw=conegold!75,line width=.16pt] (2.414214,0) -- (2.131371,0.6828427) -- (2.131371,-0.6828427) -- cycle;
					\path[fill=tipcream,draw=tipcream,line width=.16pt] (1.989949,1.024264) -- (1.648528,0.6828427) -- (2.131371,0.6828427) -- cycle;
					\path[fill=collarteal!65,draw=collarteal!65,line width=.16pt] (1.648528,0.6828427) -- (1.165685,1.165685) -- (1.648528,0) -- cycle;
					\path[fill=diskblue!35,draw=diskblue!35,line width=.16pt] (1.165685,1.165685) -- (0,0) -- (1.648528,0) -- cycle;
					\path[fill=rectcoral!75,draw=rectcoral!75,line width=.16pt] (1.989949,-1.024264) -- (1.648528,-0.6828427) -- (1.165685,-1.165685) -- (0.6828427,-1.648528) -- (1.024264,-1.989949) -- cycle;
					\path[fill=conegold!75,draw=conegold!75,line width=.16pt] (1.707107,-1.707107) -- (1.989949,-1.024264) -- (1.024264,-1.989949) -- cycle;
					\path[fill=tipcream,draw=tipcream,line width=.16pt] (2.131371,-0.6828427) -- (1.648528,-0.6828427) -- (1.989949,-1.024264) -- cycle;
					\path[fill=collarteal!65,draw=collarteal!65,line width=.16pt] (1.648528,-0.6828427) -- (1.648528,0) -- (1.165685,-1.165685) -- cycle;
					\path[fill=diskblue!35,draw=diskblue!35,line width=.16pt] (1.648528,0) -- (0,0) -- (1.165685,-1.165685) -- cycle;
					\path[fill=rectcoral!75,draw=rectcoral!75,line width=.16pt] (0.6828427,-2.131371) -- (0.6828427,-1.648528) -- (0,-1.648528) -- (-0.6828427,-1.648528) -- (-0.6828427,-2.131371) -- cycle;
					\path[fill=conegold!75,draw=conegold!75,line width=.16pt] (0,-2.414214) -- (0.6828427,-2.131371) -- (-0.6828427,-2.131371) -- cycle;
					\path[fill=tipcream,draw=tipcream,line width=.16pt] (1.024264,-1.989949) -- (0.6828427,-1.648528) -- (0.6828427,-2.131371) -- cycle;
					\path[fill=collarteal!65,draw=collarteal!65,line width=.16pt] (0.6828427,-1.648528) -- (1.165685,-1.165685) -- (0,-1.648528) -- cycle;
					\path[fill=diskblue!35,draw=diskblue!35,line width=.16pt] (1.165685,-1.165685) -- (0,0) -- (0,-1.648528) -- cycle;
					\path[fill=rectcoral!75,draw=rectcoral!75,line width=.16pt] (-1.024264,-1.989949) -- (-0.6828427,-1.648528) -- (-1.165685,-1.165685) -- (-1.648528,-0.6828427) -- (-1.989949,-1.024264) -- cycle;
					\path[fill=conegold!75,draw=conegold!75,line width=.16pt] (-1.707107,-1.707107) -- (-1.024264,-1.989949) -- (-1.989949,-1.024264) -- cycle;
					\path[fill=tipcream,draw=tipcream,line width=.16pt] (-0.6828427,-2.131371) -- (-0.6828427,-1.648528) -- (-1.024264,-1.989949) -- cycle;
					\path[fill=collarteal!65,draw=collarteal!65,line width=.16pt] (-0.6828427,-1.648528) -- (0,-1.648528) -- (-1.165685,-1.165685) -- cycle;
					\path[fill=diskblue!35,draw=diskblue!35,line width=.16pt] (0,-1.648528) -- (0,0) -- (-1.165685,-1.165685) -- cycle;
					\path[fill=rectcoral!75,draw=rectcoral!75,line width=.16pt] (-2.131371,-0.6828427) -- (-1.648528,-0.6828427) -- (-1.648528,0) -- (-1.648528,0.6828427) -- (-2.131371,0.6828427) -- cycle;
					\path[fill=conegold!75,draw=conegold!75,line width=.16pt] (-2.414214,0) -- (-2.131371,-0.6828427) -- (-2.131371,0.6828427) -- cycle;
					\path[fill=tipcream,draw=tipcream,line width=.16pt] (-1.989949,-1.024264) -- (-1.648528,-0.6828427) -- (-2.131371,-0.6828427) -- cycle;
					\path[fill=collarteal!65,draw=collarteal!65,line width=.16pt] (-1.648528,-0.6828427) -- (-1.165685,-1.165685) -- (-1.648528,0) -- cycle;
					\path[fill=diskblue!35,draw=diskblue!35,line width=.16pt] (-1.165685,-1.165685) -- (0,0) -- (-1.648528,0) -- cycle;
					\path[fill=rectcoral!75,draw=rectcoral!75,line width=.16pt] (-1.989949,1.024264) -- (-1.648528,0.6828427) -- (-1.165685,1.165685) -- (-0.6828427,1.648528) -- (-1.024264,1.989949) -- cycle;
					\path[fill=conegold!75,draw=conegold!75,line width=.16pt] (-1.707107,1.707107) -- (-1.989949,1.024264) -- (-1.024264,1.989949) -- cycle;
					\path[fill=tipcream,draw=tipcream,line width=.16pt] (-2.131371,0.6828427) -- (-1.648528,0.6828427) -- (-1.989949,1.024264) -- cycle;
					\path[fill=collarteal!65,draw=collarteal!65,line width=.16pt] (-1.648528,0.6828427) -- (-1.648528,0) -- (-1.165685,1.165685) -- cycle;
					\path[fill=diskblue!35,draw=diskblue!35,line width=.16pt] (-1.648528,0) -- (0,0) -- (-1.165685,1.165685) -- cycle;
					\path[fill=rectcoral!75,draw=rectcoral!75,line width=.16pt] (-0.6828427,2.131371) -- (-0.6828427,1.648528) -- (0,1.648528) -- (0.6828427,1.648528) -- (0.6828427,2.131371) -- cycle;
					\path[fill=conegold!75,draw=conegold!75,line width=.16pt] (0,2.414214) -- (-0.6828427,2.131371) -- (0.6828427,2.131371) -- cycle;
					\path[fill=tipcream,draw=tipcream,line width=.16pt] (-1.024264,1.989949) -- (-0.6828427,1.648528) -- (-0.6828427,2.131371) -- cycle;
					\path[fill=collarteal!65,draw=collarteal!65,line width=.16pt] (-0.6828427,1.648528) -- (-1.165685,1.165685) -- (0,1.648528) -- cycle;
					\path[fill=diskblue!35,draw=diskblue!35,line width=.16pt] (-1.165685,1.165685) -- (0,0) -- (0,1.648528) -- cycle;
					\draw[surfaceedge] (0.6828427,1.648528) -- (1.024264,1.989949);
					\draw[surfaceedge] (0.6828427,1.648528) -- (1.165685,1.165685);
					\draw[surfaceedge] (1.165685,1.165685) -- (1.648528,0.6828427);
					\draw[surfaceedge] (1.648528,0.6828427) -- (1.989949,1.024264);
					\draw[surfaceedge] (1.024264,1.989949) -- (1.989949,1.024264);
					\draw[boundaryedge] (1.024264,1.989949) -- (1.707107,1.707107);
					\draw[boundaryedge] (1.707107,1.707107) -- (1.989949,1.024264);
					\draw[surfaceedge] (0.6828427,1.648528) -- (0.6828427,2.131371);
					\draw[boundaryedge] (0.6828427,2.131371) -- (1.024264,1.989949);
					\draw[surfaceedge] (0,1.648528) -- (0.6828427,1.648528);
					\draw[surfaceedge] (0,1.648528) -- (1.165685,1.165685);
					\draw[surfaceedge] (0,0) -- (0,1.648528);
					\draw[surfaceedge] (0,0) -- (1.165685,1.165685);
					\draw[surfaceedge] (1.648528,0.6828427) -- (2.131371,0.6828427);
					\draw[surfaceedge] (1.648528,0) -- (1.648528,0.6828427);
					\draw[surfaceedge] (1.648528,-0.6828427) -- (1.648528,0);
					\draw[surfaceedge] (1.648528,-0.6828427) -- (2.131371,-0.6828427);
					\draw[surfaceedge] (2.131371,-0.6828427) -- (2.131371,0.6828427);
					\draw[boundaryedge] (2.131371,0.6828427) -- (2.414214,0);
					\draw[boundaryedge] (2.131371,-0.6828427) -- (2.414214,0);
					\draw[boundaryedge] (1.989949,1.024264) -- (2.131371,0.6828427);
					\draw[surfaceedge] (1.165685,1.165685) -- (1.648528,0);
					\draw[surfaceedge] (0,0) -- (1.648528,0);
					\draw[surfaceedge] (1.648528,-0.6828427) -- (1.989949,-1.024264);
					\draw[surfaceedge] (1.165685,-1.165685) -- (1.648528,-0.6828427);
					\draw[surfaceedge] (0.6828427,-1.648528) -- (1.165685,-1.165685);
					\draw[surfaceedge] (0.6828427,-1.648528) -- (1.024264,-1.989949);
					\draw[surfaceedge] (1.024264,-1.989949) -- (1.989949,-1.024264);
					\draw[boundaryedge] (1.707107,-1.707107) -- (1.989949,-1.024264);
					\draw[boundaryedge] (1.024264,-1.989949) -- (1.707107,-1.707107);
					\draw[boundaryedge] (1.989949,-1.024264) -- (2.131371,-0.6828427);
					\draw[surfaceedge] (1.165685,-1.165685) -- (1.648528,0);
					\draw[surfaceedge] (0,0) -- (1.165685,-1.165685);
					\draw[surfaceedge] (0.6828427,-2.131371) -- (0.6828427,-1.648528);
					\draw[surfaceedge] (0,-1.648528) -- (0.6828427,-1.648528);
					\draw[surfaceedge] (-0.6828427,-1.648528) -- (0,-1.648528);
					\draw[surfaceedge] (-0.6828427,-2.131371) -- (-0.6828427,-1.648528);
					\draw[surfaceedge] (-0.6828427,-2.131371) -- (0.6828427,-2.131371);
					\draw[boundaryedge] (0,-2.414214) -- (0.6828427,-2.131371);
					\draw[boundaryedge] (-0.6828427,-2.131371) -- (0,-2.414214);
					\draw[boundaryedge] (0.6828427,-2.131371) -- (1.024264,-1.989949);
					\draw[surfaceedge] (0,-1.648528) -- (1.165685,-1.165685);
					\draw[surfaceedge] (0,-1.648528) -- (0,0);
					\draw[surfaceedge] (-1.024264,-1.989949) -- (-0.6828427,-1.648528);
					\draw[surfaceedge] (-1.165685,-1.165685) -- (-0.6828427,-1.648528);
					\draw[surfaceedge] (-1.648528,-0.6828427) -- (-1.165685,-1.165685);
					\draw[surfaceedge] (-1.989949,-1.024264) -- (-1.648528,-0.6828427);
					\draw[surfaceedge] (-1.989949,-1.024264) -- (-1.024264,-1.989949);
					\draw[boundaryedge] (-1.707107,-1.707107) -- (-1.024264,-1.989949);
					\draw[boundaryedge] (-1.989949,-1.024264) -- (-1.707107,-1.707107);
					\draw[boundaryedge] (-1.024264,-1.989949) -- (-0.6828427,-2.131371);
					\draw[surfaceedge] (-1.165685,-1.165685) -- (0,-1.648528);
					\draw[surfaceedge] (-1.165685,-1.165685) -- (0,0);
					\draw[surfaceedge] (-2.131371,-0.6828427) -- (-1.648528,-0.6828427);
					\draw[surfaceedge] (-1.648528,-0.6828427) -- (-1.648528,0);
					\draw[surfaceedge] (-1.648528,0) -- (-1.648528,0.6828427);
					\draw[surfaceedge] (-2.131371,0.6828427) -- (-1.648528,0.6828427);
					\draw[surfaceedge] (-2.131371,-0.6828427) -- (-2.131371,0.6828427);
					\draw[boundaryedge] (-2.414214,0) -- (-2.131371,-0.6828427);
					\draw[boundaryedge] (-2.414214,0) -- (-2.131371,0.6828427);
					\draw[boundaryedge] (-2.131371,-0.6828427) -- (-1.989949,-1.024264);
					\draw[surfaceedge] (-1.648528,0) -- (-1.165685,-1.165685);
					\draw[surfaceedge] (-1.648528,0) -- (0,0);
					\draw[surfaceedge] (-1.989949,1.024264) -- (-1.648528,0.6828427);
					\draw[surfaceedge] (-1.648528,0.6828427) -- (-1.165685,1.165685);
					\draw[surfaceedge] (-1.165685,1.165685) -- (-0.6828427,1.648528);
					\draw[surfaceedge] (-1.024264,1.989949) -- (-0.6828427,1.648528);
					\draw[surfaceedge] (-1.989949,1.024264) -- (-1.024264,1.989949);
					\draw[boundaryedge] (-1.989949,1.024264) -- (-1.707107,1.707107);
					\draw[boundaryedge] (-1.707107,1.707107) -- (-1.024264,1.989949);
					\draw[boundaryedge] (-2.131371,0.6828427) -- (-1.989949,1.024264);
					\draw[surfaceedge] (-1.648528,0) -- (-1.165685,1.165685);
					\draw[surfaceedge] (-1.165685,1.165685) -- (0,0);
					\draw[surfaceedge] (-0.6828427,1.648528) -- (-0.6828427,2.131371);
					\draw[surfaceedge] (-0.6828427,1.648528) -- (0,1.648528);
					\draw[surfaceedge] (-0.6828427,2.131371) -- (0.6828427,2.131371);
					\draw[boundaryedge] (-0.6828427,2.131371) -- (0,2.414214);
					\draw[boundaryedge] (0,2.414214) -- (0.6828427,2.131371);
					\draw[boundaryedge] (-1.024264,1.989949) -- (-0.6828427,2.131371);
					\draw[surfaceedge] (-1.165685,1.165685) -- (0,1.648528);
			\end{tikzpicture}}
		\end{minipage}\hfill
		\begin{minipage}[t]{.61\textwidth}\centering
			\parbox[c][4.36cm][c]{\linewidth}{\centering
				\begin{tikzpicture}[scale=1.24,line join=round,line cap=round]
					\path[use as bounding box] (-2.481951,-0.801502) rectangle (2.481951,0.801502);
					\path[fill=rectcoral!75,draw=rectcoral!75,line width=.16pt] (1.882013,0.08095021) -- (1.415623,0.1408337) -- (1.592356,0.02629967) -- (1.706815,-0.04787744) -- (1.954616,0.03389924) -- cycle;
					\path[fill=conegold!75,draw=conegold!75,line width=.16pt] (2.138204,-0.1356281) -- (2.331951,-0.2994176) -- (2.235479,-0.1481179) -- cycle;
					\path[fill=tipcream,draw=tipcream,line width=.16pt] (1.657045,0.2711341) -- (1.415623,0.1408337) -- (1.882013,0.08095021) -- cycle;
					\path[fill=collarteal!65,draw=collarteal!65,line width=.16pt] (1.415623,0.1408337) -- (0.8242641,0.2817051) -- (1.592356,0.02629967) -- cycle;
					\path[fill=diskblue!35,draw=diskblue!35,line width=.16pt] (0.8242641,0.2817051) -- (0,0) -- (1.592356,0.02629967) -- cycle;
					\path[fill=rectcoral!75,draw=rectcoral!75,line width=.16pt] (1.541988,-0.1022722) -- (1.427667,-0.1093387) -- (1.086246,-0.130443) -- (1.5044,-0.1146767) -- cycle;
					\path[fill=conegold!75,draw=conegold!75,line width=.16pt] (1.5044,-0.1146767) -- (2.187243,-0.07246808) -- (2.09077,0.07883162) -- cycle;
					\path[fill=tipcream,draw=tipcream,line width=.16pt] (1.868198,-0.1009598) -- (2.235479,-0.1481179) -- (2.197493,-0.08854388) -- cycle;
					\path[fill=tipcream,draw=tipcream,line width=.16pt] (2.187243,-0.07246808) -- (1.822088,-0.0950394) -- (1.868198,-0.1009598) -- (2.197493,-0.08854388) -- cycle;
					\path[fill=collarteal!65,draw=collarteal!65,line width=.16pt] (1.706815,-0.04787744) -- (1.592356,0.02629967) -- (1.427667,-0.1093387) -- (1.541988,-0.1022722) -- cycle;
					\path[fill=diskblue!35,draw=diskblue!35,line width=.16pt] (1.592356,0.02629967) -- (0,0) -- (1.427667,-0.1093387) -- cycle;
					\path[fill=rectcoral!75,draw=rectcoral!75,line width=.16pt] (1.211215,-0.2114308) -- (1.086246,-0.130443) -- (0.4266705,-0.2765097) -- (-0.2329049,-0.4225764) -- (-0.107936,-0.5035642) -- cycle;
					\path[fill=conegold!75,draw=conegold!75,line width=.16pt] (0.6248444,-0.4049391) -- (1.211215,-0.2114308) -- (-0.107936,-0.5035642) -- cycle;
					\path[fill=tipcream,draw=tipcream,line width=.16pt] (1.5044,-0.1146767) -- (1.086246,-0.130443) -- (1.211215,-0.2114308) -- cycle;
					\path[fill=diskblue!35,draw=diskblue!35,line width=.16pt] (1.427667,-0.1093387) -- (0,0) -- (0.27834,-0.1803821) -- cycle;
					\path[fill=diskblue!35,draw=diskblue!35,line width=.16pt] (1.086246,-0.130443) -- (0.27834,-0.1803821) -- (0.4266705,-0.2765097) -- cycle;
					\path[fill=rectcoral!75,draw=rectcoral!75,line width=.16pt] (-0.4743263,-0.5528768) -- (-0.2329049,-0.4225764) -- (-0.8242641,-0.2817051) -- (-1.415623,-0.1408337) -- (-1.657045,-0.2711341) -- cycle;
					\path[fill=conegold!75,draw=conegold!75,line width=.16pt] (-1.207107,-0.651502) -- (-0.4743263,-0.5528768) -- (-1.657045,-0.2711341) -- cycle;
					\path[fill=tipcream,draw=tipcream,line width=.16pt] (-0.107936,-0.5035642) -- (-0.2329049,-0.4225764) -- (-0.4743263,-0.5528768) -- cycle;
					\path[fill=collarteal!65,draw=collarteal!65,line width=.16pt] (-0.2329049,-0.4225764) -- (0.4266705,-0.2765097) -- (-0.8242641,-0.2817051) -- cycle;
					\path[fill=diskblue!35,draw=diskblue!35,line width=.16pt] (0.4266705,-0.2765097) -- (0,0) -- (-0.8242641,-0.2817051) -- cycle;
					\path[fill=rectcoral!75,draw=rectcoral!75,line width=.16pt] (-1.882013,-0.08095021) -- (-1.415623,-0.1408337) -- (-1.592356,-0.02629967) -- (-1.769089,0.0882344) -- (-2.138204,0.1356281) -- cycle;
					\path[fill=conegold!75,draw=conegold!75,line width=.16pt] (-2.235479,0.1481179) -- (-1.882013,-0.08095021) -- (-2.331951,0.2994176) -- cycle;
					\path[fill=tipcream,draw=tipcream,line width=.16pt] (-1.657045,-0.2711341) -- (-1.415623,-0.1408337) -- (-1.882013,-0.08095021) -- cycle;
					\path[fill=collarteal!65,draw=collarteal!65,line width=.16pt] (-1.415623,-0.1408337) -- (-0.8242641,-0.2817051) -- (-1.592356,-0.02629967) -- cycle;
					\path[fill=diskblue!35,draw=diskblue!35,line width=.16pt] (-0.8242641,-0.2817051) -- (0,0) -- (-1.592356,-0.02629967) -- cycle;
					\path[fill=rectcoral!75,draw=rectcoral!75,line width=.16pt] (-1.769089,0.0882344) -- (-1.427667,0.1093387) -- (-1.364008,0.1199701) -- (-1.5044,0.1146767) -- (-1.822088,0.0950394) -- cycle;
					\path[fill=conegold!75,draw=conegold!75,line width=.16pt] (-2.107914,-0.05194508) -- (-2.09077,-0.07883162) -- (-2.03692,-0.06106052) -- cycle;
					\path[fill=conegold!75,draw=conegold!75,line width=.16pt] (-2.187243,0.07246808) -- (-2.107914,-0.05194508) -- (-2.03692,-0.06106052) -- (-1.954616,-0.03389924) -- (-2.122513,0.07490869) -- cycle;
					\path[fill=tipcream,draw=tipcream,line width=.16pt] (-2.187243,0.07246808) -- (-2.122513,0.07490869) -- (-2.235479,0.1481179) -- cycle;
					\path[fill=collarteal!65,draw=collarteal!65,line width=.16pt] (-1.769089,0.0882344) -- (-1.592356,-0.02629967) -- (-1.427667,0.1093387) -- cycle;
					\path[fill=diskblue!35,draw=diskblue!35,line width=.16pt] (-1.592356,-0.02629967) -- (0,0) -- (-1.427667,0.1093387) -- cycle;
					\path[fill=rectcoral!75,draw=rectcoral!75,line width=.16pt] (-1.211215,0.2114308) -- (-1.130309,0.1589988) -- (-0.4266705,0.2765097) -- (0.2329049,0.4225764) -- (0.107936,0.5035642) -- cycle;
					\path[fill=conegold!75,draw=conegold!75,line width=.16pt] (-0.6248444,0.4049391) -- (-1.211215,0.2114308) -- (0.107936,0.5035642) -- cycle;
					\path[fill=tipcream,draw=tipcream,line width=.16pt] (-1.5044,0.1146767) -- (-1.364008,0.1199701) -- (-1.130309,0.1589988) -- (-1.211215,0.2114308) -- cycle;
					\path[fill=diskblue!35,draw=diskblue!35,line width=.16pt] (-1.427667,0.1093387) -- (0,0) -- (-0.4266705,0.2765097) -- cycle;
					\path[fill=rectcoral!75,draw=rectcoral!75,line width=.16pt] (0.4743263,0.5528768) -- (0.2329049,0.4225764) -- (0.8242641,0.2817051) -- (1.415623,0.1408337) -- (1.657045,0.2711341) -- cycle;
					\path[fill=conegold!75,draw=conegold!75,line width=.16pt] (1.207107,0.651502) -- (0.4743263,0.5528768) -- (1.657045,0.2711341) -- cycle;
					\path[fill=tipcream,draw=tipcream,line width=.16pt] (0.107936,0.5035642) -- (0.2329049,0.4225764) -- (0.4743263,0.5528768) -- cycle;
					\path[fill=collarteal!65,draw=collarteal!65,line width=.16pt] (0.2329049,0.4225764) -- (-0.4266705,0.2765097) -- (0.8242641,0.2817051) -- cycle;
					\path[fill=diskblue!35,draw=diskblue!35,line width=.16pt] (-0.4266705,0.2765097) -- (0,0) -- (0.8242641,0.2817051) -- cycle;
					\draw[surfaceedge] (1.415623,0.1408337) -- (1.882013,0.08095021);
					\draw[surfaceedge] (1.415623,0.1408337) -- (1.592356,0.02629967);
					\draw[surfaceedge] (1.592356,0.02629967) -- (1.706815,-0.04787744);
					\draw[surfaceedge] (1.822088,-0.0950394) -- (2.235479,-0.1481179);
					\draw[surfaceedge] (1.882013,0.08095021) -- (1.954616,0.03389924);
					\draw[boundaryedge] (2.138204,-0.1356281) -- (2.331951,-0.2994176);
					\draw[boundaryedge] (2.331951,-0.2994176) -- (2.235479,-0.1481179);
					\draw[surfaceedge] (1.415623,0.1408337) -- (1.657045,0.2711341);
					\draw[boundaryedge] (1.657045,0.2711341) -- (1.882013,0.08095021);
					\draw[surfaceedge] (0.8242641,0.2817051) -- (1.415623,0.1408337);
					\draw[surfaceedge] (0.8242641,0.2817051) -- (1.592356,0.02629967);
					\draw[surfaceedge] (0,0) -- (0.8242641,0.2817051);
					\draw[surfaceedge] (0,0) -- (1.592356,0.02629967);
					\draw[surfaceedge] (1.427667,-0.1093387) -- (1.541988,-0.1022722);
					\draw[surfaceedge] (1.086246,-0.130443) -- (1.427667,-0.1093387);
					\draw[surfaceedge] (1.086246,-0.130443) -- (1.5044,-0.1146767);
					\draw[surfaceedge] (1.5044,-0.1146767) -- (2.187243,-0.07246808);
					\draw[boundaryedge] (2.187243,-0.07246808) -- (2.09077,0.07883162);
					\draw[boundaryedge] (1.5044,-0.1146767) -- (2.09077,0.07883162);
					\draw[boundaryedge] (2.235479,-0.1481179) -- (2.187243,-0.07246808);
					\draw[surfaceedge] (1.592356,0.02629967) -- (1.427667,-0.1093387);
					\draw[surfaceedge] (0,0) -- (1.427667,-0.1093387);
					\draw[surfaceedge] (1.086246,-0.130443) -- (1.211215,-0.2114308);
					\draw[surfaceedge] (0.4266705,-0.2765097) -- (1.086246,-0.130443);
					\draw[surfaceedge] (-0.2329049,-0.4225764) -- (0.4266705,-0.2765097);
					\draw[surfaceedge] (-0.2329049,-0.4225764) -- (-0.107936,-0.5035642);
					\draw[surfaceedge] (-0.107936,-0.5035642) -- (1.211215,-0.2114308);
					\draw[boundaryedge] (0.6248444,-0.4049391) -- (1.211215,-0.2114308);
					\draw[boundaryedge] (-0.107936,-0.5035642) -- (0.6248444,-0.4049391);
					\draw[boundaryedge] (1.211215,-0.2114308) -- (1.5044,-0.1146767);
					\draw[surfaceedge] (0.4266705,-0.2765097) -- (0.4266705,-0.2765097);
					\draw[surfaceedge] (0,0) -- (0.4266705,-0.2765097);
					\draw[surfaceedge] (-0.4743263,-0.5528768) -- (-0.2329049,-0.4225764);
					\draw[surfaceedge] (-0.8242641,-0.2817051) -- (-0.2329049,-0.4225764);
					\draw[surfaceedge] (-1.415623,-0.1408337) -- (-0.8242641,-0.2817051);
					\draw[surfaceedge] (-1.657045,-0.2711341) -- (-1.415623,-0.1408337);
					\draw[surfaceedge] (-1.657045,-0.2711341) -- (-0.4743263,-0.5528768);
					\draw[boundaryedge] (-1.207107,-0.651502) -- (-0.4743263,-0.5528768);
					\draw[boundaryedge] (-1.657045,-0.2711341) -- (-1.207107,-0.651502);
					\draw[boundaryedge] (-0.4743263,-0.5528768) -- (-0.107936,-0.5035642);
					\draw[surfaceedge] (-0.8242641,-0.2817051) -- (0.4266705,-0.2765097);
					\draw[surfaceedge] (-0.8242641,-0.2817051) -- (0,0);
					\draw[surfaceedge] (-1.882013,-0.08095021) -- (-1.415623,-0.1408337);
					\draw[surfaceedge] (-1.592356,-0.02629967) -- (-1.415623,-0.1408337);
					\draw[surfaceedge] (-1.769089,0.0882344) -- (-1.592356,-0.02629967);
					\draw[surfaceedge] (-2.138204,0.1356281) -- (-1.769089,0.0882344);
					\draw[surfaceedge] (-2.235479,0.1481179) -- (-1.882013,-0.08095021);
					\draw[boundaryedge] (-2.331951,0.2994176) -- (-1.882013,-0.08095021);
					\draw[boundaryedge] (-2.235479,0.1481179) -- (-2.331951,0.2994176);
					\draw[boundaryedge] (-1.882013,-0.08095021) -- (-1.657045,-0.2711341);
					\draw[surfaceedge] (-1.592356,-0.02629967) -- (-0.8242641,-0.2817051);
					\draw[surfaceedge] (-1.592356,-0.02629967) -- (0,0);
					\draw[surfaceedge] (-2.187243,0.07246808) -- (-2.122513,0.07490869);
					\draw[surfaceedge] (-1.769089,0.0882344) -- (-1.769089,0.0882344);
					\draw[surfaceedge] (-1.769089,0.0882344) -- (-1.427667,0.1093387);
					\draw[surfaceedge] (-1.427667,0.1093387) -- (-1.427667,0.1093387);
					\draw[surfaceedge] (-1.5044,0.1146767) -- (-1.364008,0.1199701);
					\draw[surfaceedge] (-1.822088,0.0950394) -- (-1.5044,0.1146767);
					\draw[boundaryedge] (-2.09077,-0.07883162) -- (-2.187243,0.07246808);
					\draw[boundaryedge] (-2.09077,-0.07883162) -- (-1.954616,-0.03389924);
					\draw[boundaryedge] (-2.187243,0.07246808) -- (-2.235479,0.1481179);
					\draw[surfaceedge] (-1.427667,0.1093387) -- (-1.592356,-0.02629967);
					\draw[surfaceedge] (-1.427667,0.1093387) -- (0,0);
					\draw[surfaceedge] (-1.211215,0.2114308) -- (-1.130309,0.1589988);
					\draw[surfaceedge] (-0.4266705,0.2765097) -- (-0.4266705,0.2765097);
					\draw[surfaceedge] (-0.4266705,0.2765097) -- (0.2329049,0.4225764);
					\draw[surfaceedge] (0.107936,0.5035642) -- (0.2329049,0.4225764);
					\draw[surfaceedge] (-1.211215,0.2114308) -- (0.107936,0.5035642);
					\draw[boundaryedge] (-1.211215,0.2114308) -- (-0.6248444,0.4049391);
					\draw[boundaryedge] (-0.6248444,0.4049391) -- (0.107936,0.5035642);
					\draw[boundaryedge] (-1.5044,0.1146767) -- (-1.211215,0.2114308);
					\draw[surfaceedge] (-1.427667,0.1093387) -- (-0.4266705,0.2765097);
					\draw[surfaceedge] (-0.4266705,0.2765097) -- (0,0);
					\draw[surfaceedge] (0.2329049,0.4225764) -- (0.4743263,0.5528768);
					\draw[surfaceedge] (0.2329049,0.4225764) -- (0.8242641,0.2817051);
					\draw[surfaceedge] (0.4743263,0.5528768) -- (1.657045,0.2711341);
					\draw[boundaryedge] (0.4743263,0.5528768) -- (1.207107,0.651502);
					\draw[boundaryedge] (1.207107,0.651502) -- (1.657045,0.2711341);
					\draw[boundaryedge] (0.107936,0.5035642) -- (0.4743263,0.5528768);
					\draw[surfaceedge] (-0.4266705,0.2765097) -- (0.8242641,0.2817051);
			\end{tikzpicture}}
		\end{minipage}
		\caption{Euclidean orthogonal projections of the disk $D_4$ onto $\operatorname{span}(a,b)$ (left) and $\operatorname{span}(a,b,c)$, shown in an oblique view (right).}
		\label{fig:n4}
	\end{figure}
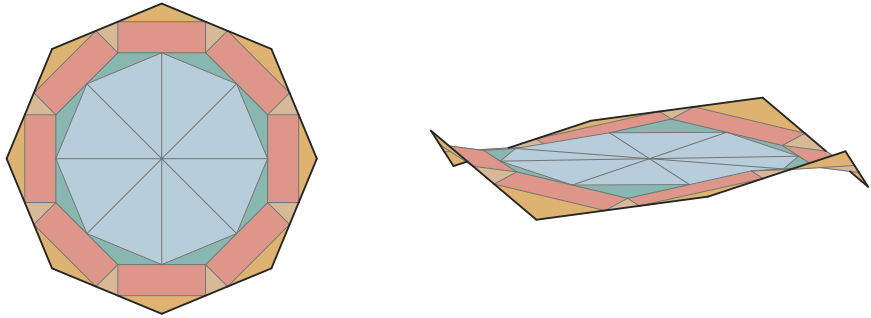
	
	\begin{figure}[htbp]
		\centering
		\begin{minipage}{.49\textwidth}\centering
			\begin{tikzpicture}[scale=.96,font=\footnotesize,line join=round,line cap=round]
				\path[use as bounding box] (-2.65,-2.65) rectangle (2.4,2.45);
				\path[fill=rectcoral!75,draw=rectcoral!75,line width=.15pt] (0.6,1.6) -- (0.4,1.2) -- (0.8,1.2) -- (1.2,1.2) -- (1.4,1.6) -- cycle;
				\path[fill=conegold!75,draw=conegold!75,line width=.15pt] (1,2) -- (0.6,1.6) -- (1.4,1.6) -- cycle;
				\path[fill=tipcream,draw=tipcream,line width=.15pt] (0.4,1.4) -- (0.4,1.2) -- (0.6,1.6) -- cycle;
				\path[fill=collarteal!65,draw=collarteal!65,line width=.15pt] (0.4,1.2) -- (0,0.8) -- (0.8,1.2) -- cycle;
				\path[fill=diskblue!35,draw=diskblue!35,line width=.15pt] (0,0.8) -- (0,0) -- (0.8,1.2) -- cycle;
				\path[fill=rectcoral!75,draw=rectcoral!75,line width=.15pt] (1.6,1.4) -- (1.2,1.2) -- (1.2,0.8) -- (1.2,0.4) -- (1.6,0.6) -- cycle;
				\path[fill=conegold!75,draw=conegold!75,line width=.15pt] (2,1) -- (1.6,1.4) -- (1.6,0.6) -- cycle;
				\path[fill=tipcream,draw=tipcream,line width=.15pt] (1.4,1.6) -- (1.2,1.2) -- (1.6,1.4) -- cycle;
				\path[fill=collarteal!65,draw=collarteal!65,line width=.15pt] (1.2,1.2) -- (0.8,1.2) -- (1.2,0.8) -- cycle;
				\path[fill=diskblue!35,draw=diskblue!35,line width=.15pt] (0.8,1.2) -- (0,0) -- (1.2,0.8) -- cycle;
				\path[fill=rectcoral!75,draw=rectcoral!75,line width=.15pt] (1.4,0.4) -- (1.2,0.4) -- (0.8,0) -- (0.4,-0.4) -- (0.6,-0.4) -- cycle;
				\path[fill=tipcream,draw=tipcream,line width=.15pt] (1.6,0.6) -- (1.2,0.4) -- (1.4,0.4) -- cycle;
				\path[fill=collarteal!65,draw=collarteal!65,line width=.15pt] (1.2,0.4) -- (1.2,0.8) -- (0.8,0) -- cycle;
				\path[fill=diskblue!35,draw=diskblue!35,line width=.15pt] (1.2,0.8) -- (0,0) -- (0.8,0) -- cycle;
				\path[fill=rectcoral!75,draw=rectcoral!75,line width=.15pt] (0.4,-0.6) -- (0.4,-0.4) -- (0,-0.8) -- (-0.4,-1.2) -- (-0.4,-1.4) -- cycle;
				\path[fill=tipcream,draw=tipcream,line width=.15pt] (0.6,-0.4) -- (0.4,-0.4) -- (0.4,-0.6) -- cycle;
				\path[fill=diskblue!35,draw=diskblue!35,line width=.15pt] (0.8,0) -- (0,0) -- (0,-0.8) -- cycle;
				\path[fill=rectcoral!75,draw=rectcoral!75,line width=.15pt] (-0.6,-1.6) -- (-0.4,-1.2) -- (-0.8,-1.2) -- (-1.2,-1.2) -- (-1.4,-1.6) -- cycle;
				\path[fill=conegold!75,draw=conegold!75,line width=.15pt] (-1,-2) -- (-0.6,-1.6) -- (-1.4,-1.6) -- cycle;
				\path[fill=tipcream,draw=tipcream,line width=.15pt] (-0.4,-1.4) -- (-0.4,-1.2) -- (-0.6,-1.6) -- cycle;
				\path[fill=collarteal!65,draw=collarteal!65,line width=.15pt] (-0.4,-1.2) -- (0,-0.8) -- (-0.8,-1.2) -- cycle;
				\path[fill=diskblue!35,draw=diskblue!35,line width=.15pt] (0,-0.8) -- (0,0) -- (-0.8,-1.2) -- cycle;
				\path[fill=rectcoral!75,draw=rectcoral!75,line width=.15pt] (-1.6,-1.4) -- (-1.2,-1.2) -- (-1.2,-0.8) -- (-1.2,-0.4) -- (-1.6,-0.6) -- cycle;
				\path[fill=conegold!75,draw=conegold!75,line width=.15pt] (-2,-1) -- (-1.6,-1.4) -- (-1.6,-0.6) -- cycle;
				\path[fill=tipcream,draw=tipcream,line width=.15pt] (-1.4,-1.6) -- (-1.2,-1.2) -- (-1.6,-1.4) -- cycle;
				\path[fill=collarteal!65,draw=collarteal!65,line width=.15pt] (-1.2,-1.2) -- (-0.8,-1.2) -- (-1.2,-0.8) -- cycle;
				\path[fill=diskblue!35,draw=diskblue!35,line width=.15pt] (-0.8,-1.2) -- (0,0) -- (-1.2,-0.8) -- cycle;
				\path[fill=rectcoral!75,draw=rectcoral!75,line width=.15pt] (-1.4,-0.4) -- (-1.2,-0.4) -- (-0.8,0) -- (-0.4,0.4) -- (-0.6,0.4) -- cycle;
				\path[fill=tipcream,draw=tipcream,line width=.15pt] (-1.6,-0.6) -- (-1.2,-0.4) -- (-1.4,-0.4) -- cycle;
				\path[fill=collarteal!65,draw=collarteal!65,line width=.15pt] (-1.2,-0.4) -- (-1.2,-0.8) -- (-0.8,0) -- cycle;
				\path[fill=diskblue!35,draw=diskblue!35,line width=.15pt] (-1.2,-0.8) -- (0,0) -- (-0.8,0) -- cycle;
				\path[fill=rectcoral!75,draw=rectcoral!75,line width=.15pt] (-0.4,0.6) -- (-0.4,0.4) -- (0,0.8) -- (0.4,1.2) -- (0.4,1.4) -- cycle;
				\path[fill=tipcream,draw=tipcream,line width=.15pt] (-0.6,0.4) -- (-0.4,0.4) -- (-0.4,0.6) -- cycle;
				\path[fill=diskblue!35,draw=diskblue!35,line width=.15pt] (-0.8,0) -- (0,0) -- (0,0.8) -- cycle;
				\projectiongrid
				\draw[black!40,line width=.25pt] (-2,-1) -- (-1.6,-1.4);
				\draw[black!40,line width=.25pt] (-2,-1) -- (-1.6,-0.6);
				\draw[black!40,line width=.25pt] (-1.6,-1.4) -- (-1.6,-0.6);
				\draw[black!40,line width=.25pt] (-1.6,-1.4) -- (-1.4,-1.6);
				\draw[black!40,line width=.25pt] (-1.6,-1.4) -- (-1.2,-1.2);
				\draw[black!40,line width=.25pt] (-1.6,-0.6) -- (-1.4,-0.4);
				\draw[black!40,line width=.25pt] (-1.6,-0.6) -- (-1.2,-0.4);
				\draw[black!40,line width=.25pt] (-1.4,-1.6) -- (-1.2,-1.2);
				\draw[black!40,line width=.25pt] (-1.4,-1.6) -- (-1,-2);
				\draw[black!40,line width=.25pt] (-1.4,-1.6) -- (-0.6,-1.6);
				\draw[black!40,line width=.25pt] (-1.4,-0.4) -- (-1.2,-0.4);
				\draw[black!40,line width=.25pt] (-1.4,-0.4) -- (-1,0);
				\draw[black!40,line width=.25pt] (-1.4,-0.4) -- (-0.6,0.4);
				\draw[black!40,line width=.25pt] (-1.2,-1.2) -- (-1.2,-0.8);
				\draw[black!40,line width=.25pt] (-1.2,-1.2) -- (-0.8,-1.2);
				\draw[black!40,line width=.25pt] (-1.2,-0.8) -- (-1.2,-0.4);
				\draw[black!40,line width=.25pt] (-1.2,-0.8) -- (-0.8,-1.2);
				\draw[black!40,line width=.25pt] (-1.2,-0.8) -- (-0.8,0);
				\draw[black!40,line width=.25pt] (-1.2,-0.8) -- (0,0);
				\draw[black!40,line width=.25pt] (-1.2,-0.4) -- (-0.8,0);
				\draw[black!40,line width=.25pt] (-1,-2) -- (-0.6,-1.6);
				\draw[black!40,line width=.25pt] (-1,0) -- (-0.6,0.4);
				\draw[black!40,line width=.25pt] (-0.8,-1.2) -- (-0.4,-1.2);
				\draw[black!40,line width=.25pt] (-0.8,-1.2) -- (0,-0.8);
				\draw[black!40,line width=.25pt] (-0.8,-1.2) -- (0,0);
				\draw[black!40,line width=.25pt] (-0.8,0) -- (-0.4,0.4);
				\draw[black!40,line width=.25pt] (-0.8,0) -- (0,0);
				\draw[black!40,line width=.25pt] (-0.8,0) -- (0,0.8);
				\draw[black!40,line width=.25pt] (-0.6,-1.6) -- (-0.4,-1.4);
				\draw[black!40,line width=.25pt] (-0.6,-1.6) -- (-0.4,-1.2);
				\draw[black!40,line width=.25pt] (-0.6,0.4) -- (-0.4,0.4);
				\draw[black!40,line width=.25pt] (-0.6,0.4) -- (-0.4,0.6);
				\draw[black!40,line width=.25pt] (-0.4,-1.4) -- (-0.4,-1.2);
				\draw[black!40,line width=.25pt] (-0.4,-1.4) -- (0,-1);
				\draw[black!40,line width=.25pt] (-0.4,-1.4) -- (0.4,-0.6);
				\draw[black!40,line width=.25pt] (-0.4,-1.2) -- (0,-0.8);
				\draw[black!40,line width=.25pt] (-0.4,0.4) -- (-0.4,0.6);
				\draw[black!40,line width=.25pt] (-0.4,0.4) -- (0,0.8);
				\draw[black!40,line width=.25pt] (-0.4,0.6) -- (0,1);
				\draw[black!40,line width=.25pt] (-0.4,0.6) -- (0.4,1.4);
				\draw[black!40,line width=.25pt] (0,-1) -- (0.4,-0.6);
				\draw[black!40,line width=.25pt] (0,-0.8) -- (0,0);
				\draw[black!40,line width=.25pt] (0,-0.8) -- (0.4,-0.4);
				\draw[black!40,line width=.25pt] (0,-0.8) -- (0.8,0);
				\draw[black!40,line width=.25pt] (0,0) -- (0,0.8);
				\draw[black!40,line width=.25pt] (0,0) -- (0.8,0);
				\draw[black!40,line width=.25pt] (0,0) -- (0.8,1.2);
				\draw[black!40,line width=.25pt] (0,0) -- (1.2,0.8);
				\draw[black!40,line width=.25pt] (0,0.8) -- (0.4,1.2);
				\draw[black!40,line width=.25pt] (0,0.8) -- (0.8,1.2);
				\draw[black!40,line width=.25pt] (0,1) -- (0.4,1.4);
				\draw[black!40,line width=.25pt] (0.4,-0.6) -- (0.4,-0.4);
				\draw[black!40,line width=.25pt] (0.4,-0.6) -- (0.6,-0.4);
				\draw[black!40,line width=.25pt] (0.4,-0.4) -- (0.6,-0.4);
				\draw[black!40,line width=.25pt] (0.4,-0.4) -- (0.8,0);
				\draw[black!40,line width=.25pt] (0.4,1.2) -- (0.4,1.4);
				\draw[black!40,line width=.25pt] (0.4,1.2) -- (0.6,1.6);
				\draw[black!40,line width=.25pt] (0.4,1.2) -- (0.8,1.2);
				\draw[black!40,line width=.25pt] (0.4,1.4) -- (0.6,1.6);
				\draw[black!40,line width=.25pt] (0.6,-0.4) -- (1,0);
				\draw[black!40,line width=.25pt] (0.6,-0.4) -- (1.4,0.4);
				\draw[black!40,line width=.25pt] (0.6,1.6) -- (1,2);
				\draw[black!40,line width=.25pt] (0.6,1.6) -- (1.4,1.6);
				\draw[black!40,line width=.25pt] (0.8,0) -- (1.2,0.4);
				\draw[black!40,line width=.25pt] (0.8,0) -- (1.2,0.8);
				\draw[black!40,line width=.25pt] (0.8,1.2) -- (1.2,0.8);
				\draw[black!40,line width=.25pt] (0.8,1.2) -- (1.2,1.2);
				\draw[black!40,line width=.25pt] (1,0) -- (1.4,0.4);
				\draw[black!40,line width=.25pt] (1,2) -- (1.4,1.6);
				\draw[black!40,line width=.25pt] (1.2,0.4) -- (1.2,0.8);
				\draw[black!40,line width=.25pt] (1.2,0.4) -- (1.4,0.4);
				\draw[black!40,line width=.25pt] (1.2,0.4) -- (1.6,0.6);
				\draw[black!40,line width=.25pt] (1.2,0.8) -- (1.2,1.2);
				\draw[black!40,line width=.25pt] (1.2,1.2) -- (1.4,1.6);
				\draw[black!40,line width=.25pt] (1.2,1.2) -- (1.6,1.4);
				\draw[black!40,line width=.25pt] (1.4,0.4) -- (1.6,0.6);
				\draw[black!40,line width=.25pt] (1.4,1.6) -- (1.6,1.4);
				\draw[black!40,line width=.25pt] (1.6,0.6) -- (1.6,1.4);
				\draw[black!40,line width=.25pt] (1.6,0.6) -- (2,1);
				\draw[black!40,line width=.25pt] (1.6,1.4) -- (2,1);
				\draw[black!85,line width=.65pt] (2,1) -- (1,2) -- (-2,-1) -- (-1,-2) -- cycle;
				\projectionframe{x_1}{x_2}
			\end{tikzpicture}
		\end{minipage}\hfill
		\begin{minipage}{.49\textwidth}\centering
			\begin{tikzpicture}[scale=.96,font=\footnotesize,line join=round,line cap=round]
				\path[use as bounding box] (-2.65,-2.65) rectangle (2.4,2.45);
				\path[fill=rectcoral!75,draw=rectcoral!75,line width=.15pt] (0.6,1.4) -- (0.4,1.2) -- (0.8,0.8) -- (1.2,0.4) -- (1.4,0.6) -- cycle;
				\path[fill=tipcream,draw=tipcream,line width=.15pt] (0.4,1.6) -- (0.4,1.2) -- (0.6,1.4) -- cycle;
				\path[fill=collarteal!65,draw=collarteal!65,line width=.15pt] (0.4,1.2) -- (0,1.2) -- (0.8,0.8) -- cycle;
				\path[fill=diskblue!35,draw=diskblue!35,line width=.15pt] (0,1.2) -- (0,0) -- (0.8,0.8) -- cycle;
				\path[fill=rectcoral!75,draw=rectcoral!75,line width=.15pt] (1.6,0.4) -- (1.2,0.4) -- (1.2,0) -- (1.2,-0.4) -- (1.6,-0.4) -- cycle;
				\path[fill=conegold!75,draw=conegold!75,line width=.15pt] (2,0) -- (1.6,0.4) -- (1.6,-0.4) -- cycle;
				\path[fill=tipcream,draw=tipcream,line width=.15pt] (1.4,0.6) -- (1.2,0.4) -- (1.6,0.4) -- cycle;
				\path[fill=collarteal!65,draw=collarteal!65,line width=.15pt] (1.2,0.4) -- (0.8,0.8) -- (1.2,0) -- cycle;
				\path[fill=diskblue!35,draw=diskblue!35,line width=.15pt] (0.8,0.8) -- (0,0) -- (1.2,0) -- cycle;
				\path[fill=rectcoral!75,draw=rectcoral!75,line width=.15pt] (1.4,-0.6) -- (1.2,-0.4) -- (0.8,-0.8) -- (0.4,-1.2) -- (0.6,-1.4) -- cycle;
				\path[fill=tipcream,draw=tipcream,line width=.15pt] (1.6,-0.4) -- (1.2,-0.4) -- (1.4,-0.6) -- cycle;
				\path[fill=collarteal!65,draw=collarteal!65,line width=.15pt] (1.2,-0.4) -- (1.2,0) -- (0.8,-0.8) -- cycle;
				\path[fill=diskblue!35,draw=diskblue!35,line width=.15pt] (1.2,0) -- (0,0) -- (0.8,-0.8) -- cycle;
				\path[fill=rectcoral!75,draw=rectcoral!75,line width=.15pt] (0.4,-1.6) -- (0.4,-1.2) -- (0,-1.2) -- (-0.4,-1.2) -- (-0.4,-1.6) -- cycle;
				\path[fill=conegold!75,draw=conegold!75,line width=.15pt] (0,-2) -- (0.4,-1.6) -- (-0.4,-1.6) -- cycle;
				\path[fill=tipcream,draw=tipcream,line width=.15pt] (0.6,-1.4) -- (0.4,-1.2) -- (0.4,-1.6) -- cycle;
				\path[fill=collarteal!65,draw=collarteal!65,line width=.15pt] (0.4,-1.2) -- (0.8,-0.8) -- (0,-1.2) -- cycle;
				\path[fill=diskblue!35,draw=diskblue!35,line width=.15pt] (0.8,-0.8) -- (0,0) -- (0,-1.2) -- cycle;
				\path[fill=rectcoral!75,draw=rectcoral!75,line width=.15pt] (-0.6,-1.4) -- (-0.4,-1.2) -- (-0.8,-0.8) -- (-1.2,-0.4) -- (-1.4,-0.6) -- cycle;
				\path[fill=tipcream,draw=tipcream,line width=.15pt] (-0.4,-1.6) -- (-0.4,-1.2) -- (-0.6,-1.4) -- cycle;
				\path[fill=collarteal!65,draw=collarteal!65,line width=.15pt] (-0.4,-1.2) -- (0,-1.2) -- (-0.8,-0.8) -- cycle;
				\path[fill=diskblue!35,draw=diskblue!35,line width=.15pt] (0,-1.2) -- (0,0) -- (-0.8,-0.8) -- cycle;
				\path[fill=rectcoral!75,draw=rectcoral!75,line width=.15pt] (-1.6,-0.4) -- (-1.2,-0.4) -- (-1.2,0) -- (-1.2,0.4) -- (-1.6,0.4) -- cycle;
				\path[fill=conegold!75,draw=conegold!75,line width=.15pt] (-2,0) -- (-1.6,-0.4) -- (-1.6,0.4) -- cycle;
				\path[fill=tipcream,draw=tipcream,line width=.15pt] (-1.4,-0.6) -- (-1.2,-0.4) -- (-1.6,-0.4) -- cycle;
				\path[fill=collarteal!65,draw=collarteal!65,line width=.15pt] (-1.2,-0.4) -- (-0.8,-0.8) -- (-1.2,0) -- cycle;
				\path[fill=diskblue!35,draw=diskblue!35,line width=.15pt] (-0.8,-0.8) -- (0,0) -- (-1.2,0) -- cycle;
				\path[fill=rectcoral!75,draw=rectcoral!75,line width=.15pt] (-1.4,0.6) -- (-1.2,0.4) -- (-0.8,0.8) -- (-0.4,1.2) -- (-0.6,1.4) -- cycle;
				\path[fill=tipcream,draw=tipcream,line width=.15pt] (-1.6,0.4) -- (-1.2,0.4) -- (-1.4,0.6) -- cycle;
				\path[fill=collarteal!65,draw=collarteal!65,line width=.15pt] (-1.2,0.4) -- (-1.2,0) -- (-0.8,0.8) -- cycle;
				\path[fill=diskblue!35,draw=diskblue!35,line width=.15pt] (-1.2,0) -- (0,0) -- (-0.8,0.8) -- cycle;
				\path[fill=rectcoral!75,draw=rectcoral!75,line width=.15pt] (-0.4,1.6) -- (-0.4,1.2) -- (0,1.2) -- (0.4,1.2) -- (0.4,1.6) -- cycle;
				\path[fill=conegold!75,draw=conegold!75,line width=.15pt] (0,2) -- (-0.4,1.6) -- (0.4,1.6) -- cycle;
				\path[fill=tipcream,draw=tipcream,line width=.15pt] (-0.6,1.4) -- (-0.4,1.2) -- (-0.4,1.6) -- cycle;
				\path[fill=collarteal!65,draw=collarteal!65,line width=.15pt] (-0.4,1.2) -- (-0.8,0.8) -- (0,1.2) -- cycle;
				\path[fill=diskblue!35,draw=diskblue!35,line width=.15pt] (-0.8,0.8) -- (0,0) -- (0,1.2) -- cycle;
				\projectiongrid
				\draw[black!40,line width=.25pt] (-2,0) -- (-1.6,-0.4);
				\draw[black!40,line width=.25pt] (-2,0) -- (-1.6,0.4);
				\draw[black!40,line width=.25pt] (-1.6,-0.4) -- (-1.6,0.4);
				\draw[black!40,line width=.25pt] (-1.6,-0.4) -- (-1.4,-0.6);
				\draw[black!40,line width=.25pt] (-1.6,-0.4) -- (-1.2,-0.4);
				\draw[black!40,line width=.25pt] (-1.6,0.4) -- (-1.4,0.6);
				\draw[black!40,line width=.25pt] (-1.6,0.4) -- (-1.2,0.4);
				\draw[black!40,line width=.25pt] (-1.4,-0.6) -- (-1.2,-0.4);
				\draw[black!40,line width=.25pt] (-1.4,-0.6) -- (-1,-1);
				\draw[black!40,line width=.25pt] (-1.4,-0.6) -- (-0.6,-1.4);
				\draw[black!40,line width=.25pt] (-1.4,0.6) -- (-1.2,0.4);
				\draw[black!40,line width=.25pt] (-1.4,0.6) -- (-1,1);
				\draw[black!40,line width=.25pt] (-1.4,0.6) -- (-0.6,1.4);
				\draw[black!40,line width=.25pt] (-1.2,-0.4) -- (-1.2,0);
				\draw[black!40,line width=.25pt] (-1.2,-0.4) -- (-0.8,-0.8);
				\draw[black!40,line width=.25pt] (-1.2,0) -- (-1.2,0.4);
				\draw[black!40,line width=.25pt] (-1.2,0) -- (-0.8,-0.8);
				\draw[black!40,line width=.25pt] (-1.2,0) -- (-0.8,0.8);
				\draw[black!40,line width=.25pt] (-1.2,0) -- (0,0);
				\draw[black!40,line width=.25pt] (-1.2,0.4) -- (-0.8,0.8);
				\draw[black!40,line width=.25pt] (-1,-1) -- (-0.6,-1.4);
				\draw[black!40,line width=.25pt] (-1,1) -- (-0.6,1.4);
				\draw[black!40,line width=.25pt] (-0.8,-0.8) -- (-0.4,-1.2);
				\draw[black!40,line width=.25pt] (-0.8,-0.8) -- (0,-1.2);
				\draw[black!40,line width=.25pt] (-0.8,-0.8) -- (0,0);
				\draw[black!40,line width=.25pt] (-0.8,0.8) -- (-0.4,1.2);
				\draw[black!40,line width=.25pt] (-0.8,0.8) -- (0,0);
				\draw[black!40,line width=.25pt] (-0.8,0.8) -- (0,1.2);
				\draw[black!40,line width=.25pt] (-0.6,-1.4) -- (-0.4,-1.6);
				\draw[black!40,line width=.25pt] (-0.6,-1.4) -- (-0.4,-1.2);
				\draw[black!40,line width=.25pt] (-0.6,1.4) -- (-0.4,1.2);
				\draw[black!40,line width=.25pt] (-0.6,1.4) -- (-0.4,1.6);
				\draw[black!40,line width=.25pt] (-0.4,-1.6) -- (-0.4,-1.2);
				\draw[black!40,line width=.25pt] (-0.4,-1.6) -- (0,-2);
				\draw[black!40,line width=.25pt] (-0.4,-1.6) -- (0.4,-1.6);
				\draw[black!40,line width=.25pt] (-0.4,-1.2) -- (0,-1.2);
				\draw[black!40,line width=.25pt] (-0.4,1.2) -- (-0.4,1.6);
				\draw[black!40,line width=.25pt] (-0.4,1.2) -- (0,1.2);
				\draw[black!40,line width=.25pt] (-0.4,1.6) -- (0,2);
				\draw[black!40,line width=.25pt] (-0.4,1.6) -- (0.4,1.6);
				\draw[black!40,line width=.25pt] (0,-2) -- (0.4,-1.6);
				\draw[black!40,line width=.25pt] (0,-1.2) -- (0,0);
				\draw[black!40,line width=.25pt] (0,-1.2) -- (0.4,-1.2);
				\draw[black!40,line width=.25pt] (0,-1.2) -- (0.8,-0.8);
				\draw[black!40,line width=.25pt] (0,0) -- (0,1.2);
				\draw[black!40,line width=.25pt] (0,0) -- (0.8,-0.8);
				\draw[black!40,line width=.25pt] (0,0) -- (0.8,0.8);
				\draw[black!40,line width=.25pt] (0,0) -- (1.2,0);
				\draw[black!40,line width=.25pt] (0,1.2) -- (0.4,1.2);
				\draw[black!40,line width=.25pt] (0,1.2) -- (0.8,0.8);
				\draw[black!40,line width=.25pt] (0,2) -- (0.4,1.6);
				\draw[black!40,line width=.25pt] (0.4,-1.6) -- (0.4,-1.2);
				\draw[black!40,line width=.25pt] (0.4,-1.6) -- (0.6,-1.4);
				\draw[black!40,line width=.25pt] (0.4,-1.2) -- (0.6,-1.4);
				\draw[black!40,line width=.25pt] (0.4,-1.2) -- (0.8,-0.8);
				\draw[black!40,line width=.25pt] (0.4,1.2) -- (0.4,1.6);
				\draw[black!40,line width=.25pt] (0.4,1.2) -- (0.6,1.4);
				\draw[black!40,line width=.25pt] (0.4,1.2) -- (0.8,0.8);
				\draw[black!40,line width=.25pt] (0.4,1.6) -- (0.6,1.4);
				\draw[black!40,line width=.25pt] (0.6,-1.4) -- (1,-1);
				\draw[black!40,line width=.25pt] (0.6,-1.4) -- (1.4,-0.6);
				\draw[black!40,line width=.25pt] (0.6,1.4) -- (1,1);
				\draw[black!40,line width=.25pt] (0.6,1.4) -- (1.4,0.6);
				\draw[black!40,line width=.25pt] (0.8,-0.8) -- (1.2,-0.4);
				\draw[black!40,line width=.25pt] (0.8,-0.8) -- (1.2,0);
				\draw[black!40,line width=.25pt] (0.8,0.8) -- (1.2,0);
				\draw[black!40,line width=.25pt] (0.8,0.8) -- (1.2,0.4);
				\draw[black!40,line width=.25pt] (1,-1) -- (1.4,-0.6);
				\draw[black!40,line width=.25pt] (1,1) -- (1.4,0.6);
				\draw[black!40,line width=.25pt] (1.2,-0.4) -- (1.2,0);
				\draw[black!40,line width=.25pt] (1.2,-0.4) -- (1.4,-0.6);
				\draw[black!40,line width=.25pt] (1.2,-0.4) -- (1.6,-0.4);
				\draw[black!40,line width=.25pt] (1.2,0) -- (1.2,0.4);
				\draw[black!40,line width=.25pt] (1.2,0.4) -- (1.4,0.6);
				\draw[black!40,line width=.25pt] (1.2,0.4) -- (1.6,0.4);
				\draw[black!40,line width=.25pt] (1.4,-0.6) -- (1.6,-0.4);
				\draw[black!40,line width=.25pt] (1.4,0.6) -- (1.6,0.4);
				\draw[black!40,line width=.25pt] (1.6,-0.4) -- (1.6,0.4);
				\draw[black!40,line width=.25pt] (1.6,-0.4) -- (2,0);
				\draw[black!40,line width=.25pt] (1.6,0.4) -- (2,0);
				\draw[black!85,line width=.65pt] (2,0) -- (0,2) -- (-2,0) -- (0,-2) -- cycle;
				\projectionframe{x_1}{x_3}
			\end{tikzpicture}
		\end{minipage}
		\caption{Coordinate projections of $D_4/h$ onto $(x_1,x_2)$ and
			$(x_1,x_3)$.}
		\label{fig:coordinate-projections}
	\end{figure}
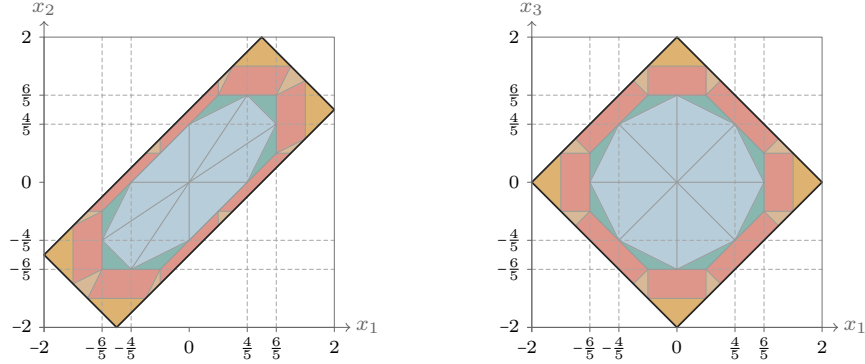
	
	This shows that $\Fill(S_4) \le 11 \pi^2/20$. To prove the matching lower bound, we work on the normalized set $E_4$.
	By \eqref{eq:projection-rectangle},
	\[
	R_1=\{(s,t):|s-t|\le1,\ |s+t|\le3\},\qquad
	R_2=\{(s,t):|s-t|\le2,\ |s+t|\le2\}.
	\]
	On these rectangles define
	\begin{align*}
		f_1(s,t)&\coloneqq
		\begin{cases}
			1,&\min(|s|,|t|)>6/5,\\
			1/2,&\max(|s|,|t|)>6/5\ge\min(|s|,|t|)>4/5,\\
			0,&\text{otherwise},
		\end{cases}\\
		f_2(s,t)&\coloneqq
		\begin{cases}
			1/2,&\max(|s|,|t|)\le6/5\text{ and }\min(|s|,|t|)\le4/5,\\
			0,&\text{otherwise}.
		\end{cases}
	\end{align*}
	Extend $f_1$ and $f_2$ by zero outside $R_1$ and $R_2$, respectively. The normalized calibration on $E_4$ is
	\begin{equation}\label{eq:n4-form}
		\omega_x\coloneqq\sum_{i=1}^4 f_1(x_i,x_{i+1})\,dx_i\wedge dx_{i+1}
		+f_2(x_1,x_3)\,dx_1\wedge dx_3
		+f_2(x_2,x_4)\,dx_2\wedge dx_4,
	\end{equation}
	where $x_5\coloneqq-x_1$.
	
	\begin{figure}[htbp]
		\centering
		\begin{minipage}{.49\textwidth}\centering
			\begin{tikzpicture}[scale=.96,font=\footnotesize]
				\path[use as bounding box] (-2.65,-2.65) rectangle (2.4,2.45);
				\fill[heatnone] (2,1) -- (1,2) -- (-2,-1) -- (-1,-2) -- cycle;
				\begin{scope}
					\clip (2,1) -- (1,2) -- (-2,-1) -- (-1,-2) -- cycle;
					\fill[heatfull] (-1.2,-1.2) -- (-2,-1.2) -- (-2,-2) -- (-1.2,-2) -- cycle;
					\fill[heathalf] (-0.8,-1.2) -- (-1.2,-1.2) -- (-1.2,-2) -- (-0.8,-2) -- cycle;
					\fill[heathalf] (-1.2,-0.8) -- (-2,-0.8) -- (-2,-1.2) -- (-1.2,-1.2) -- cycle;
					\fill[heatfull] (1.2,1.2) -- (2,1.2) -- (2,2) -- (1.2,2) -- cycle;
					\fill[heathalf] (0.8,1.2) -- (1.2,1.2) -- (1.2,2) -- (0.8,2) -- cycle;
					\fill[heathalf] (1.2,0.8) -- (2,0.8) -- (2,1.2) -- (1.2,1.2) -- cycle;
				\end{scope}
				\projectiongrid
				\draw[black!65,line width=.5pt] (2,1) -- (1,2) -- (-2,-1) -- (-1,-2) -- cycle;
				\projectionframe{s}{t}
			\end{tikzpicture}
		\end{minipage}\hfill
		\begin{minipage}{.49\textwidth}\centering
			\begin{tikzpicture}[scale=.96,font=\footnotesize]
				\path[use as bounding box] (-2.65,-2.65) rectangle (2.4,2.45);
				\fill[heatnone] (2,0) -- (0,2) -- (-2,0) -- (0,-2) -- cycle;
				\begin{scope}
					\clip (2,0) -- (0,2) -- (-2,0) -- (0,-2) -- cycle;
					\fill[heathalf] (-0.8,-1.2) -- (0.8,-1.2) -- (0.8,-0.8) -- (1.2,-0.8) -- (1.2,0.8) -- (0.8,0.8) -- (0.8,1.2) -- (-0.8,1.2) -- (-0.8,0.8) -- (-1.2,0.8) -- (-1.2,-0.8) -- (-0.8,-0.8) -- cycle;
				\end{scope}
				\projectiongrid
				\draw[black!65,line width=.5pt] (2,0) -- (0,2) -- (-2,0) -- (0,-2) -- cycle;
				\projectionframe{s}{t}
			\end{tikzpicture}
		\end{minipage}
		\par\smallskip
		\begin{tikzpicture}[font=\footnotesize]
			\filldraw[fill=heatnone,draw=black!30] (0,0) rectangle (.45,.22);
			\node[right] at (.5,.11) {$0$};
			\filldraw[fill=heathalf,draw=black!30] (1.5,0) rectangle (1.95,.22);
			\node[right] at (2,.11) {$\tfrac12$};
			\filldraw[fill=heatfull,draw=black!30] (3.2,0) rectangle (3.65,.22);
			\node[right] at (3.7,.11) {$1$};
		\end{tikzpicture}
		\caption{The two coefficient functions of the four-dimensional calibration: $f_1$ on $R_1$ (left) and $f_2$ on $R_2$ (right).}
		\label{fig:calibration}
	\end{figure}
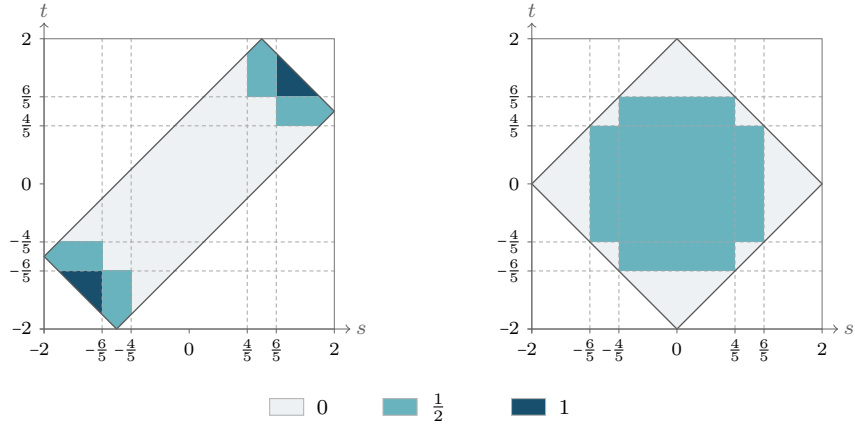
	
	\begin{Prop}\label{prop:n4-inequality}
		The coefficients of \eqref{eq:n4-form} satisfy, for every $x\in E_4$,
		\[
		\sum_{i=1}^4 f_1(x_i,x_{i+1})+f_2(x_1,x_3)+f_2(x_2,x_4)\le1.
		\]
	\end{Prop}
	\begin{proof}
		Put $m_i\coloneqq|x_i|$. The defining inequalities of $E_4$ imply
		\[
		m_1+m_3=\max(|x_1-x_3|,|x_1+x_3|)\le2,
		\qquad m_2+m_4\le2.
		\]
		Classify a magnitude as low if it is at most $4/5$, middle if it belongs to
		$(4/5,6/5]$, and high if it exceeds $6/5$. A high magnitude $m_i$ forces $m_{i\pm2}$ to be low, so there are at most two high magnitudes.
		
		If there are none, all four $f_1$ terms vanish and the two $f_2$ terms sum to at most one. If there is exactly one, say $m_1$ is high, then $m_3$ is low. Only the two $f_1$ terms incident to the first coordinate and $f_2(x_2,x_4)$ can contribute. If both $m_2,m_4$ are middle, their contributions are $1/2+1/2+0$; if exactly one is middle, they are $1/2+0+1/2$; and if both are low, they are $0+0+1/2$. Finally, if there are two high magnitudes, they are adjacent and the other two are low. Precisely their common $f_1$ term can contribute, and it equals one.
	\end{proof}
	
	We give the integrals explicitly. The support of $f_2$ is a square of side $8/5$ with four rectangular arms of dimensions $(8/5)\times(2/5)$, all contained
	in $R_2$. Hence
	\[
	\int_{R_2}f_2\,d\cL^2
	=\frac12\left[\left(\frac85\right)^2
	+4\frac85\frac25\right]=\frac{64}{25}.
	\]
	The region where $f_1=1$ consists of two right triangles of leg length $3/5$, with total area $9/25$. The region where $f_1=1/2$ consists of four congruent pieces, each of area
	\[
	\int_{4/5}^{1}(s-1/5)\,ds+
	\int_{1}^{6/5}(9/5-s)\,ds=\frac7{25}.
	\]
	It follows that
	\[
	\int_{R_1}f_1\,d\cL^2=\frac9{25}+\frac12\frac{28}{25}=\frac{23}{25}.
	\]
	Given a filling $T$ of $S_4$, retract it onto $hE_4$ using Proposition~\ref{prop:hull}, without increasing its mass. Propositions~\ref{prop:calibration} and~\ref{prop:n4-inequality}, with the weights rescaled by $h$, give
	\[
	\bM(T)\ge h^2\left(4\frac{23}{25}+2\frac{64}{25}\right)
	=\frac{44}{5}h^2=\frac{11}{20}\pi^2.
	\]
	Together with~\eqref{eq:n4-upper}, this proves~\eqref{eq:main-values} for $n=4$.
	
	\subsection{Lower bounds for fillings of the circle}
	
	
	The calibration of $D_4$ inside the injective hull of its boundary $S_4$ gives the following lower bounds.
	
	\begin{Cor}[Metric fillings]\label{cor:circle-currents}
		Let $X$ be a complete metric space containing an oriented isometric copy $S$ of $\mathbb S^1$. Every metric $2$-current $T$ in $X$ with $\partial T=\curr S$ satisfies
		\[
		\bM(T)\ge\frac{11}{20}\pi^2.
		\]
		More generally, for any complete normed Abelian group $G$ and $g\in G$, every rectifiable $G$-chain $T$ in $X$ with $\partial T=g\curr S$ satisfies
		\begin{equation}\label{eq:circle-group-bound}
			\bM(T)\ge |g|\,\frac{11}{20}\pi^2,
		\end{equation}
		with mass$^\ast$ as in Remark~\ref{rem:coefficients}.
	\end{Cor}
	\begin{proof}
		Identify $S$ with $\mathbb S^1$. By injectivity, $\rho_4:S\to h_4E_4$ extends to a $1$-Lipschitz map $F:X\to h_4E_4$. Thus $F_\#T$ has boundary $\Gamma_4$, or $g\Gamma_4$ in the coefficient case, and no greater mass. Apply the calibration of $D_4$ and Remark~\ref{rem:coefficients}.
	\end{proof}
	
	\begin{Cor}[Riemannian fillings]\label{cor:lower_bound}
		Let $\Sigma$ be a compact Riemannian surface whose boundary, with the distance induced from $\Sigma$, is isometric to $\mathbb S^1$. Then
		\[
		\Area(\Sigma)\ge\frac{11}{20}\pi^2=5.428282\ldots.
		\]
	\end{Cor}
	\begin{proof}
		If $\Sigma$ is orientable, its fundamental integral current has mass $\Area(\Sigma)$ and boundary $\curr{\partial\Sigma}$, with the induced orientation. Apply Corollary~\ref{cor:circle-currents} with $X=\Sigma$. If $\Sigma$ is nonorientable, use its fundamental $\Z/2\Z$-chain instead.
	\end{proof}
	
	\section{Linear programs}\label{sec:lp}
	
	Fix $n\ge2$. Consider the following optimization problem over bounded Borel functions
	$f_k:R_k\to\R$, $1 \le k < n$, extended by zero outside $R_k$:
	\begin{align}
		\text{maximize}\quad&
		A_n(f)\coloneqq
		\sum_{k=1}^{n-1}(n-k)\int_{R_k}f_k\,d\cL^2
		\label{eq:lp-objective}\\
		\text{subject to}\quad&
		\sum_{1\le i<j\le n}f_{j-i}(x_i,x_j)\le1
		&&\text{for all }x\in E_n,\label{eq:lp-pointwise}\\
		&f_k(s,t)\ge0
		&&\text{for all }1\le k<n,\ (s,t)\in R_k.\notag
	\end{align}
	
	This is an infinite linear program. Every feasible family gives
	\begin{equation}\label{eq:lp-bound}
		\Fill(S_n)\ge\left(\frac{\pi}{n}\right)^2 A_n(f),
	\end{equation}
	and the same lower bound for the fillings of $\mathbb S^1$ in
	Corollaries~\ref{cor:circle-currents} and~\ref{cor:lower_bound}.
	
	Averaging over the signed cyclic shifts and reflections of $E_n$
	allows us to impose the symmetries
	\[
	f_k(s,t)=f_k(t,s)=f_k(-s,-t),
	\qquad
	f_{n-k}(s,t)=f_k(t,-s)
	\]
	without changing the supremum. These identities reduce the number of
	independent functions and the domains on which they must be specified.
	
	To obtain a finite program, choose a symmetric grid
	\[
	-n/2=t_0<t_1<\cdots<t_N=n/2,
	\qquad t_{N-a}=-t_a,
	\]
	and put $I_a\coloneqq(t_{a-1},t_a)$. For $1\le k<n$ and
	$1\le a,b\le N$, set
	\[
	A_{k,a,b}\coloneqq
	\cL^2\bigl(R_k\cap(I_a\times I_b)\bigr),
	\qquad
	\mathcal V\coloneqq\{(k,a,b):A_{k,a,b}>0\}.
	\]
	For each $(k,a,b)\in\mathcal V$, the variable $f_{k,a,b}$ denotes the
	constant value of $f_k$ on $R_k\cap(I_a\times I_b)$.
	We set the functions to zero on all grid lines.
	
	For $\mathbf a=(a_1,\ldots,a_n)\in\{1,\ldots,N\}^n$, set
	$B_{\mathbf a}\coloneqq\prod_{i=1}^n I_{a_i}$ and define
	\begin{equation}\label{eq:lp-box-test}
		\mathcal B\coloneqq
		\bigl\{\mathbf a\in\{1,\ldots,N\}^n:
		(j-i,a_i,a_j)\in\mathcal V
		\text{ for all }1\le i<j\le n\bigr\}.
	\end{equation}
	By Lemma~\ref{lem:lp-boxes} below, this filter keeps exactly the boxes meeting $E_n$. By definition, every projected cell of a retained box automatically belongs to $\mathcal V$.
	
	The finite linear program associated with the chosen grid is
	\begin{equation}\label{eq:lp-finite}
		\begin{aligned}
			\text{maximize}\quad&
			\sum_{(k,a,b)\in\mathcal V}
			(n-k)A_{k,a,b}\,f_{k,a,b}\\
			\text{subject to}\quad&
			\sum_{1\le i<j\le n}f_{j-i,a_i,a_j}\le1
			&&\text{for all }\mathbf a\in\mathcal B,\\
			&f_{k,a,b}\ge0
			&&\text{for all }(k,a,b)\in\mathcal V.
		\end{aligned}
	\end{equation}
	Its objective equals $A_n(f)$ for the associated step functions.
	The constraints imply \eqref{eq:lp-pointwise} away from the grid hyperplanes.
	At a point on a grid hyperplane, approximate from within $E_n$ by points off all grid hyperplanes. Every nonzero term at the original point persists under sufficiently small perturbations, while the additional terms are nonnegative. Thus \eqref{eq:lp-pointwise} holds there as well, and every feasible solution gives the lower bound~\eqref{eq:lp-bound}.
	
	The symmetry reductions become linear equalities. With
	$\bar a\coloneqq N+1-a$, they read
	\[
	f_{k,a,b}=f_{k,b,a}=f_{k,\bar a,\bar b},
	\qquad
	f_{n-k,a,b}=f_{k,b,\bar a}.
	\]
	
	There is also a symmetry reduction of the constraints. The map
	\[
	\widehat\sigma_n(a_1,\ldots,a_n)
	\coloneqq(\bar a_n, a_1,\ldots,a_{n-1})
	\]
	preserves $\mathcal B$. Under the preceding equalities, the constraints
	indexed by $\mathbf a$ and $\widehat\sigma_n(\mathbf a)$ coincide.
	It therefore suffices to retain one representative of each orbit.
	Reversing the indices $(a_1,\ldots,a_n)$ gives a further symmetry of
	the constraints.
	
	
	It remains to justify the box filter.
	
	\begin{Lem}\label{lem:lp-boxes}
		\[
		\mathbf a\in\mathcal B
		\quad\Longleftrightarrow\quad
		B_{\mathbf a}\cap E_n\ne\varnothing.
		\]
	\end{Lem}
	
	\begin{proof}
		If $B_{\mathbf a}\cap E_n$ is nonempty, moving an intersection point
		slightly towards the origin gives
		$y\in B_{\mathbf a}\cap\operatorname{int}E_n$.
		Extend the coordinates by $y_{r+n}=-y_r$.
		The neighbouring inequalities are strict, so for $1\le i<j\le n$,
		\[
		\begin{aligned}
			|y_i-y_j|
			&\le\sum_{r=i}^{j-1}|y_{r+1}-y_r|<j-i,\\
			|y_i+y_j|=|y_j-y_{i+n}|
			&\le\sum_{r=j}^{i+n-1}|y_{r+1}-y_r|<n-(j-i).
		\end{aligned}
		\]
		Thus $(y_i,y_j)$ lies in the interior of $R_{j-i}$ and in the open
		rectangle $I_{a_i}\times I_{a_j}$. Their intersection therefore has positive area, and $\mathbf a\in\mathcal B$.
		
		Conversely, suppose $\mathbf a\in\mathcal B$. Write $I_{a_i}=(\ell_i,u_i)$ and extend the intervals by reflection to the indices in $\{n+1,\dots,2n\}$:
		\[
		\ell_{i+n}=-u_i,\qquad u_{i+n}=-\ell_i.
		\]
		We seek values in these intervals which change by at most $1$ between neighbouring indices on the $2n$-cycle and satisfy $y_{i+n}=-y_i$.
		
		Regard the indices $1,\ldots,2n$ as vertices of a cycle with unit edge lengths, and let
		\[
		d(p,q)\coloneqq\min\{|p-q|,2n-|p-q|\}.
		\]
		For $i<j$, we have $d(i,j)=j-i$ and $d(i,j+n)=n-(j-i)$.
		Thus \eqref{eq:projection-rectangle} reads
		\[
		R_{j-i}
		=\bigl\{(s,t): |s-t|\le d(i,j),\
		|s+t|\le d(i,j+n)\bigr\}.
		\]
		The second inequality is a difference bound between $s$ and the
		reflected value $-t$.
		
		Since every projected cell is nonempty, reflecting either or both
		coordinates shows that, for distinct nonantipodal indices $p,q$,
		there exist
		\[
		s\in(\ell_p,u_p),\qquad t\in(\ell_q,u_q),
		\qquad |s-t|\le d(p,q).
		\]
		Hence
		\[
		\ell_q<t\le s+d(p,q)<u_p+d(p,q).
		\]
		For $p=q$, the inequality $\ell_q<u_p$ is immediate.
		For $|p-q|=n$, we have $\ell_q=-u_p$ and $u_p>-n/2$,
		so $\ell_q<u_p+n$. Thus in all cases,
		\[
		\ell_q-d(p,q)<u_p
		\qquad\text{for all }1\le p,q\le2n.
		\]
		A lower bound $\ell_q$ at index $q$ forces the lower bound $\ell_q-d(p,q)$ at index $p$ for any sequence satisfying the neighbouring inequalities. Taking all these bounds into account, set
		\[
		z_p\coloneqq\max_{1\le q\le2n}\{\ell_q-d(p,q)\}.
		\]
		The preceding inequalities give $\ell_p\le z_p<u_p$. Each function $p\mapsto\ell_q-d(p,q)$ is $1$-Lipschitz with respect to $d$, so their maximum is also $1$-Lipschitz.
		
		Finally, enforce antiperiodicity by averaging:
		\[
		y_p\coloneqq\frac{z_p-z_{p+n}}2,
		\]
		with indices understood modulo $2n$. Both sequences being averaged are $1$-Lipschitz with respect to $d$, and $y_{p+n}=-y_p$. Moreover,
		\[
		z_p\in[\ell_p,u_p),
		\qquad -z_{p+n}\in(\ell_p,u_p].
		\]
		Thus $\ell_p<y_p<u_p$. Consequently, $(y_1,\ldots,y_n) \in B_{\mathbf a} \cap E_n$.
	\end{proof}
	
	For $n=4$, the grid
	\[
	(t_0,\ldots,t_5)
	=\left(-2,-\frac65,-\frac45,\frac45,\frac65,2\right)
	\]
	yields the optimal value $A_4=44/5$, recovering the lower bound $\Fill(S_4) \ge 11\pi^2/20$ established above.
	
	For $n=6$, numerical optimization on the uniform grid of mesh $1/4$
	suggests a feasible value $A_6=323/16$, corresponding to a lower
	bound of $323\pi^2/576\approx5.53452$ for $\Fill(S_6)$. We report this as numerical evidence; the bound is not proved here. Further optimization we leave to those with more compute.
	
	\section{On Question~\ref*{que:polyhedral}}\label{sec:polyhedral}
	
	For a prescribed polyhedral boundary in $\ell^n_\infty$, an
	integral mass minimizer exists by compactness and lower
	semicontinuity~\cite{AK}. Question~\ref{que:polyhedral} asks whether
	such a minimizer can always be chosen polyhedral.
	
	For $m=n$, the cone construction~\cite[4.1.11]{F} gives a
	polyhedral filling of the prescribed boundary. By the constancy
	theorem~\cite[4.1.7]{F}, this is the unique integral filling in
	$\R^n$, and hence mass minimizing.
	
	The case $m=1$ follows from Kantorovich--Rubinstein duality;
	see~\cite[Theorem~1.14]{Vi}. The least transport cost from the
	negative part of the prescribed boundary $S$ to its positive
	part equals $\sup_{\Lip(\varphi)\le1}S(\varphi)$, which bounds
	the mass of every filling from below. For integer coefficients,
	an optimal transport can be chosen as a matching of points
	repeated according to their multiplicities. Joining matched
	points by oriented straight segments gives a polyhedral filling
	attaining this bound.
	
	Both arguments apply to every norm and every Finsler volume.
	For the intermediate dimensions $1<m<n$, we return to mass$^\ast$
	in $\ell^n_\infty$. As a step towards answering
	Question~\ref{que:polyhedral}, we give a local minimality criterion
	for the union of two $m$-dimensional half-spaces with a common boundary.
	
	\begin{Prop}\label{prop:touching_subspaces}
		Let \(H_1,H_2\subset\ell^n_\infty\) be \(m\)-dimensional half-spaces with
		\[
		H_1\cap H_2=\partial H_1=\partial H_2,
		\qquad
		T\coloneqq H_1\cup H_2.
		\]
		Then \(T\) is locally mass\(^{\ast}\) minimizing under compactly supported Lipschitz deformations if and only if some coordinate projection restricts to a bijective, mass\(^{\ast}\)-preserving map
		\[
		\pi_\lambda:T\longrightarrow\ell^m_\infty.
		\]
	\end{Prop}
	
	\begin{proof}
		Write $\mu$ for mass$^\ast$ volume. Suppose such a projection exists, and set $s=(\pi_\lambda|_T)^{-1}$.
		Let $\varphi$ be a compactly supported Lipschitz deformation, and choose
		a bounded open set $U$ containing the support of $\varphi-\mathrm{id}$
		and its image under $\varphi$. The map
		$\pi_\lambda\circ\varphi\circ s$ agrees with the identity outside a compact
		set, so it has degree one and is surjective. Since $\varphi$ fixes
		$T\setminus U$ and maps $T\cap U$ into $U$, it follows that
		\[
		\pi_\lambda(T\cap U)\subset\pi_\lambda(\varphi(T)\cap U).
		\]
		Projection preserves $\mu$ on $T$ and does not increase it on rectifiable
		sets. Hence $\mu(T\cap U)\le\mu(\varphi(T)\cap U)$, as required.
		
		Conversely, suppose that \(T\) is locally minimizing. Put \(k\coloneqq m-1\) and translate the common boundary to a linear subspace \(L\). Choose a basis \(e_1,\ldots,e_k\) of \(L\), and put \(e\coloneqq e_1\wedge\cdots\wedge e_k\), with \(e\coloneqq1\) when \(k=0\). Write
		\[
		H_1=L+\mathbb R_+u,\qquad H_2=L+\mathbb R_+v
		\]
		for some $u,v\in\R^n$. If $T$ is an $m$-dimensional subspace,
		the statement is immediate. Otherwise, $u$ and $v$ are linearly
		independent modulo $L$. Define
		\[
		F(z)\coloneqq\mu(e\wedge z)
		=\max_{\lambda\in\Lambda(n,m)}
		|dx_\lambda(e\wedge z)|.
		\]
		Linearity of the coordinate forms gives
		\[
		F(u-v)
		\le
		\max_\lambda\bigl(
		|dx_\lambda(e\wedge u)|+
		|dx_\lambda(e\wedge v)|
		\bigr)
		\le F(u)+F(v).
		\]
		We claim that equality holds. Suppose instead that
		\(\delta\coloneqq F(u)+F(v)-F(u-v)>0\).
		For \(R>0\), set
		\[
		Q_R\coloneqq
		\Biggl\{\sum_{j=1}^k t_je_j:|t_j|\le R\Biggr\},
		\qquad
		\Delta\coloneqq\conv\{v,0,u\}.
		\]
		The patch
		\[
		A_R\coloneqq
		\bigl(Q_R+[0,u]\bigr)\cup
		\bigl(Q_R+[0,v]\bigr)
		\subset T
		\]
		has volume
		\[
		\mu(A_R)=(2R)^k\bigl(F(u)+F(v)\bigr).
		\]
		Consider the prism \(P_R\coloneqq Q_R+\Delta\), an affine copy of \(Q_R\times\Delta\). The patch \(A_R\) consists of its two adjacent facets corresponding to \([v,0]\) and \([0,u]\). The remaining boundary is $D_R\coloneqq B_R\cup S_R$, where
		\[
		B_R\coloneqq Q_R+[v,u]
		\quad \mbox{and} \quad
		S_R\coloneqq\partial Q_R+\Delta.
		\]
		For \(k=0\), the side \(S_R\) is absent.
		Both \(A_R\) and \(D_R\) are polyhedral \(m\)-disks, and their boundaries coincide. A piecewise affine homeomorphism \(A_R\to D_R\) fixing this boundary, together with the identity on \(T\setminus A_R\), extends to a compactly supported Lipschitz deformation of the ambient space. Thus this replacement is admissible.
		The volume of \(B_R\) is \((2R)^kF(u-v)\).
		For \(k\ge1\), each side face is the product of a \((k-1)\)-dimensional face of \(Q_R\) with the fixed triangle \(\Delta\), and hence
		\[
		\mu(S_R)=C R^{k-1},\qquad C\coloneqq\mu(S_1).
		\]
		Consequently,
		\[
		\mu(D_R)-\mu(A_R)
		\le -(2R)^k\delta+CR^{k-1}<0
		\]
		for sufficiently large \(R\). For \(k=0\), the difference is simply \(-\delta<0\). Scaling the construction about the origin places this volume-decreasing deformation in an arbitrarily small neighborhood, contradicting local minimality. Therefore
		\[
		F(u)+F(v)=F(u-v).
		\]
		Choose a coordinate index \(\lambda\) attaining the maximum defining \(F(u-v)\), and a sign \(\varepsilon\) such that
		\[
		\varepsilon\,dx_\lambda(e\wedge(u-v))=F(u-v).
		\]
		Since \(\varepsilon\,dx_\lambda(e\wedge u)\le F(u)\) and
		\(-\varepsilon\,dx_\lambda(e\wedge v)\le F(v)\), equality in their sum gives
		\[
		\varepsilon\,dx_\lambda(e\wedge u)=F(u)>0,
		\qquad
		-\varepsilon\,dx_\lambda(e\wedge v)=F(v)>0.
		\]
		Thus $\pi_\lambda$ is injective and volume-preserving on each supporting $m$-plane. The signs show that the two half-spaces project to opposite sides of $\pi_\lambda(L)$. Hence their union maps bijectively onto $\ell^m_\infty$, preserving volume.
	\end{proof}
	
	For an integral polyhedral minimizer, write the incident $m$-faces at an interior $(m-1)$-face as oriented sheets of multiplicity one. Boundary cancellation gives equally many sheets with each boundary orientation. Replacing any pair with opposite boundary orientations by a smaller filling would decrease the mass of the chain. Hence every such pair admits a common coordinate calibration.
	
	\subsection*{Origin of the examples and use of AI}
	
	The disks $D_3$ and $D_4$, together with the calibration of $D_3$,
	were found in 2022 through guessing and numerical
	surface minimization using SciPy. The calibration of $D_4$ was
	found in 2025, guided by linear programming experiments.
	Its structure suggested the general formulation in
	Section~\ref{sec:lp}. ChatGPT was subsequently used to revise the exposition and
	mathematical arguments, check calculations and references,
	and prepare the TikZ figures. The author takes responsibility
	for the mathematical content.
	
	\enlargethispage{3\baselineskip}
	
\end{document}